\documentclass[10pt,a4paper,final]{article}

\usepackage{tikz}
\usepackage{booktabs}
\usepackage{authblk}
\usepackage{silence}
\usepackage[T1]{fontenc}
\usepackage[utf8]{inputenc}
\usepackage{csquotes}
\usepackage[UKenglish]{babel}
\usepackage{paralist}
\usepackage[shortlabels]{enumitem}
\usepackage[margin=.6in]{geometry}
\usepackage[dvipsnames]{xcolor}
\usepackage{graphicx}
\usepackage{bbm}
\usepackage{amsmath}
\usepackage{amssymb}
\usepackage{amsfonts}
\usepackage{stmaryrd}
\usepackage{mathrsfs}
\usepackage{mathtools}
\usepackage{epstopdf}
\usepackage{comment}
\usepackage{xspace}
\usepackage{float}
\usepackage{stackengine}

\usepackage{amsthm}
\usepackage{hyperref}
\hypersetup{
    linktocpage=true,
    colorlinks=true,
    citecolor=darkgreen,
    linkcolor=blue,
    urlcolor=blue,
}
\usepackage[nameinlink,capitalise]{cleveref}
\usepackage{autonum}

\newcommand{\email}[1]{\href{mailto:#1}{#1}}

\usepackage{setspace}
\input{widebar.tex}

\usepackage{titlesec}
\titleformat*{\section}{\large\bfseries\sffamily}
\titlespacing*{\subsection}{0pt}{10pt}{*0.5}

\DeclarePairedDelimiter{\abs}{\lvert}{\rvert}
\DeclarePairedDelimiter{\norm}{\lVert}{\rVert}
\DeclarePairedDelimiter{\bra}{(}{)}

\DeclareMathOperator{\e}{e}

\newcommand{\range}[2]{\{#1, \ldots, #2 \}}

\renewcommand{\d}{\mathrm d}

\newcommand{\N}{\mathbf{N}}
\newcommand{\R}{\mathbf R}

\newcommand{\expect}{\mathbf{E}}

\renewcommand{\leq}{\leqslant}
\renewcommand{\geq}{\geqslant}
\renewcommand{\le}{\leqslant}
\renewcommand{\ge}{\geqslant}

\newcommand{\particleNumber}[1]{#1}
\newcommand{\xn}[1]{X^{\particleNumber{#1}}}

\newcommand{\xl}{\widebar{X}}
\newcommand{\xnl}[1]{\widebar{X}^{\particleNumber{#1}}}
\newcommand{\mfldis}{\widebar{\rho}}

\theoremstyle{plain}

\newtheorem{theorem}{Theorem}
\newtheorem{lemma}[theorem]{Lemma}
\newtheorem{corollary}[theorem]{Corollary}
\newtheorem{proposition}[theorem]{Proposition}

\newtheorem{remark}[theorem]{Remark}

\newtheorem{assumption}[theorem]{Assumption}

\crefname{lemma}{Lemma}{Lemmas}
\crefname{remark}{Remark}{Remarks}
\crefname{assumption}{Assumption}{Assumptions}
\crefname{proposition}{Proposition}{Propositions}
\crefname{section}{Section}{Sections}
\crefname{subsection}{Subsection}{Subsections}
\crefname{equation}{}{}
\Crefname{equation}{Equation}{Equations}

\newlist{lemmaenum}{enumerate}{3}
\setlist[lemmaenum]{label=(\alph*),ref=\,(\alph*)}
\crefname{lemmaenum}{Lemma}{Lemmas}
\newlist{assumpenum}{enumerate}{5}
\setlist[assumpenum]{label=(\alph*), font={\bfseries}}
\newlist{auxenum}{enumerate}{2}
\setlist[auxenum]{label=(\alph*),ref=(\alph*)}
\crefname{auxenumi}{Item}{Items}
\crefname{enumi}{}{}
\crefname{equation}{}{}
\crefname{assumpenumi}{}{}
\crefname{assumpenumii}{}{}
\Crefname{assumpenumi}{Assumption}{Assumptions}
\Crefname{assumpenumii}{Assumption}{Assumptions}
\Crefname{assumpenumii}{Assumption}{Assumptions}
\crefrangeformat{assumpenumi}{#3#1#4--#5#2#6}
\Crefname{lemmaenumi}{Part}{Parts}
\Crefname{figure}{Figure}{Figures}
\numberwithin{equation}{section}
\numberwithin{theorem}{section}

\usepackage{tikz}
\usetikzlibrary{shapes.misc}
\usetikzlibrary{positioning}

\cslet{blx@noerroretextools}\empty 
\usepackage[url=true,backend=biber,style=numeric,url=false,doi=false,isbn=false]{biblatex}
\DeclareFieldFormat{volume}{volume \textbf{#1}}
\DeclareFieldFormat[article]{volume}{\textbf{#1}}

\let\oldparagraph=\paragraph
\renewcommand\paragraph[1]{\oldparagraph{#1.}}

\newcommand{\eps}{\varepsilon}

\definecolor{darkred}{rgb}{.7,0,0}
\definecolor{darkgreen}{rgb}{.1,.7,0}

\usetikzlibrary{decorations.pathreplacing}

\definecolor{TabBlue}{HTML}{1f77b4}
\definecolor{TabOrange}{HTML}{ff7f0e}
\definecolor{TabGreen}{HTML}{2ca02c}
\definecolor{TabRed}{HTML}{d62728}
\definecolor{TabPurple}{HTML}{9467bd}
\definecolor{TabBrown}{HTML}{8c564b}
\definecolor{TabPink}{HTML}{e377c2}
\definecolor{TabGray}{HTML}{7f7f7f}
\definecolor{TabYellow}{HTML}{bcbd22}
\definecolor{TabCyan}{HTML}{17becf}

\begin{document}
\title{Uniform-in-time propagation of chaos for the Cucker--Smale model}

\author[1]{N. J. Gerber$^{a,}$}
\author[2,3]{U. Vaes$^{b,}$}
\affil[ ]{\footnotesize $^a$\email{nicolai.gerber@uni-ulm.de},
$^b$\email{urbain.vaes@inria.fr}}
\affil[1]{\footnotesize Institute of Applied Analysis, Ulm University, Germany}
\affil[2]{\footnotesize MATHERIALS project-team, Inria Paris}
\affil[3]{\footnotesize CERMICS, \'Ecole des Ponts, France}
\date{\today}

\maketitle

\newcommand{\zn}[1]{Z^{#1}}
\newcommand{\znt}[1]{\widetilde Z^{#1}}
\newcommand{\xnn}[2]{X^{#1,\particleNumber{#2}}}
\newcommand{\vnn}[2]{V^{#1,\particleNumber{#2}}}
\newcommand{\Znn}[2]{\mathcal Z^{#1, #2}}
\newcommand{\vn}[1]{V^{\particleNumber{#1}}}
\newcommand{\wn}[1]{W^{\particleNumber{#1}}}
\newcommand{\vl}{\widebar{V}}
\newcommand{\vnl}[1]{\widebar{V}^{\particleNumber{#1}}}
\newcommand{\xnt}[1]{\widetilde X^{\particleNumber{#1}}}
\newcommand{\vnt}[1]{\widetilde V^{\particleNumber{#1}}}
\newcommand{\xinitn}[1]{x^{\particleNumber{#1}}_0}
\newcommand{\vinitn}[1]{v^{\particleNumber{#1}}_0}
\newcommand{\dxn}[1]{x^{\particleNumber{#1}}}
\newcommand{\dvn}[1]{v^{\particleNumber{#1}}}
\newcommand{\lip}{L_{\psi}}
\newcommand{\meanv}{\mathfrak V}
\newcommand{\meanvl}{\widebar{\mathfrak V}}
\newcommand{\meanx}{\mathfrak X}
\newcommand{\meanxl}{\widebar{\mathfrak X}}
\newcommand{\allv}[1]{\mathcal{V}^{#1}}
\newcommand{\allvt}[1]{\widetilde{\mathcal{V}}^{#1}}
\newcommand{\allz}[1]{\mathcal{Z}^{#1}}
\newcommand{\allzt}[1]{\widetilde{\mathcal{Z}}^{#1}}
\newcommand{\allzf}[1]{\mathfrak{Z}^{#1}}
\newcommand{\allzft}[1]{\widetilde{\mathfrak{Z}}^{#1}}
\newcommand{\allx}[1]{\mathcal{X}^{#1}}
\newcommand{\ally}[1]{\mathcal{Y}^{#1}}
\newcommand{\allxt}[1]{\widetilde{\mathcal{X}}^{#1}}
\newcommand{\ev}{\Phi}
\newcommand{\evl}{\overline{\Phi}} 
\newcommand{\diameter}{\mathcal D}
\newcommand{\diametert}{\widetilde{\mathcal D}}
\newcommand{\diametermfl}{\widebar{\mathcal D}}
\newcommand{\marg}[1]{\varrho^{(#1)}}
\newcommand{\emp}[2]{\mu_{\allz{#1}_{#2}}} 
\newcommand{\empl}[2]{\widebar{\mu_{\allz{#1}_{#2}}}} 
\newcommand{\empt}[2]{\mu_{\allzt{#1}_{#2}}} 
\newcommand{\meanX}[1]{\mathcal M\bigl(\mu_{\allx{J}_{#1}}\bigr)}
\newcommand{\covx}[1]{\mathcal C\bigl(\mu_{\allx{J}_{#1}}\bigr)}
\newcommand{\covmodx}[1]{\mathcal S\bigl(\mu_{\allx{J}_{#1}}\bigr)}
\newcommand{\covy}[1]{\mathcal C\bigl(\mu_{\ally{J}_{#1}}\bigr)}
\newcommand{\covox}[1]{\mathcal C^{\circ}\bigl(\mu_{\allx{J}_{#1}}\bigr)}
\newcommand{\Psit}{\widetilde \Psi}
\newcommand{\supp}{\operatorname{supp}}

\newcommand{\revX}[1]{{\color{OliveGreen}#1}}
\newcommand{\revY}[1]{{\color{orange}#1}}
\newcommand{\additional}[1]{{\color{teal}#1}}

\AddToHook{env/lemma/begin}{\crefalias{theorem}{lemma}}
\AddToHook{env/proposition/begin}{\crefalias{theorem}{proposition}}
\AddToHook{env/remark/begin}{\crefalias{theorem}{remark}}
\AddToHook{env/definition/begin}{\crefalias{theorem}{definition}}
\AddToHook{env/example/begin}{\crefalias{theorem}{example}}
\AddToHook{env/assumption/begin}{\crefalias{theorem}{assumption}}
\AddToHook{env/corollary/begin}{\crefalias{theorem}{corollary}}

\newcommand{\qone}{\cref{question:1}}
\newcommand{\qtwo}{\cref{question:2}}
\newcommand{\qthree}{\cref{question:3}}
\newcommand{\qfour}{\cref{question:4}}
\newcommand{\real}{\R}

\begin{abstract}
    This paper presents an elementary proof of quantitative uniform-in-time propagation of chaos for the Cucker--Smale model
    under sufficiently strong interaction.
    The idea is to combine existing finite-time propagation of chaos estimates
    with existing uniform-in-time stability estimates for the interacting particle system,
    in order to obtain a uniform-in-time propagation of chaos estimate with an explicit rate of convergence in the number of particles.
    This is achieved via a method that is similar in spirit to the classical ``stability + consistency implies convergence'' approach in numerical analysis.
\end{abstract}

\section{Introduction}
\label{sec:Model}

\subsection{The Cucker--Smale model and its mean field limit}

Let $\bigl(\xinitn{j}, \vinitn{j}\bigr)_{j \in \N}$ denote independent random variables with common distribution $\mfldis_0 \in \mathcal P(\R^d \times \R^d)$.
We consider for~$J \in \N$ the Cucker--Smale interacting particle system~\cite{HaKimZhang2018uniformstab}
with random initial conditions sampled i.i.d.\ from a given probability measure $\mfldis_0$:
\begin{equation}
    \label{eq:cucker-smale}
    \left\{
    \begin{aligned}
        \frac{\d\xn{j}_t}{\d t}  & = \vn{j}_t,
        \\
        \frac{\d \vn{j}_t}{\d t} & = -\frac{K}{J} \sum_{k=1}^{J}
        \psi \Bigl( \bigl\lvert \xn{j}_t - \xn{k}_t \bigr\rvert \Bigr)
        \left( \vn{j}_t - \vn{k}_t \right),
    \end{aligned}
    \right.
    \qquad
    \left(\xn{j}_0, \vn{j}_0\right) = \left(\xinitn{j}, \vinitn{j}\right)
    \stackrel{\rm i.i.d.}{\sim } \mfldis_0.
\end{equation}
Here $K$ denotes the communication strength,
and the communication rate function $\psi\colon (0, \infty) \to (0, \infty)$ is a globally Lispchitz continuous, non-increasing function
that is bounded from above by 1.
That is to say:
\begin{equation}
    \label{eq:assumption_psi}
    \psi > 0, \qquad
    \norm{\psi}_{L^{\infty}} \leq 1, \qquad
    \norm{\psi}_{\rm Lip} < \infty, \qquad
    \bigl( \psi(r_1) - \psi(r_2) \bigr) (r_1 - r_2) \leq 0 \quad \forall r_1, r_2 > 0.
\end{equation}
This is a standing assumption throughout this paper.
An example given in~\cite{MR2860672} of a function satisfying assumption~\eqref{eq:assumption_psi} is
\(
\psi(z) = \left(1 + z^2\right)^{-\gamma},
\)
for some~$\gamma \geq 0$.
Equation~\eqref{eq:cucker-smale} is known as the \textbf{Cucker--Smale model}~\cite{cucker2007mathematics, cucker2007emergent,MR3644594},
a widely studied model to describe the evolution of flocks.

Propagation of chaos for~\eqref{eq:cucker-smale} refers to the property that, in the limit as~$J \to \infty$ in~\eqref{eq:cucker-smale},
the particles become asymptotically independent and the marginal law of a particle at some time $t\ge 0$ converges,
in a suitable sense, to the law $\mfldis_t = {\rm Law}(\xl_t, \vl_t)$ at time $t$ of the McKean--Vlasov process
\begin{equation}
    \label{eq:cucker-smale_mean_field}
    \left\{
    \begin{aligned}
        \frac{\d \xl_t}{\d t} & = \vl_t,                                                              \\
        \frac{\d \vl_t}{\d t} & = -K \left(H \star \mfldis_t\right) \left(\xl_t, \vl_t\right), \qquad
        \qquad \left(\xl_0, \vl_0\right) \sim \mfldis_0,                                              \\
        \mfldis_t             & = {\rm Law}(\xl_t, \vl_t),
    \end{aligned}
    \right.
\end{equation}
where $H(x, v) := \psi( \lvert x \rvert ) v$
and the convolution between $H \colon \R^d \times \R^d \to \R^d$ and $\mu \in \mathcal P(\R^d \times \R^d)$ is defined as
\[
    (H \star \mu)(x, v) := \iint_{\R^d \times \R^d} H(x - y, v - w) \, \mu(\d y \, \d w).
\]
The solution to~\eqref{eq:cucker-smale_mean_field} will be called \emph{the mean-field process} in the rest of this paper.
The dynamics of the mean-field law~$(\mfldis_t)_{t \geq 0}$ is governed by the following McKean--Vlasov equation:
\begin{align}
    \label{eq:mean_field_pde}
    \partial_t \mfldis_t + v \cdot \nabla_x \mfldis_t - K \nabla_v \cdot \bigl(H \star \mfldis_t \bigr) = 0, \qquad x, v \in \R^d.
\end{align}
The existence of a globally defined measure-valued solution to this equation is proved in~\cite[Section 6]{MR2536440}.
As noted in Remark~5.1 of that reference,
see also~\cite[Remark 8]{HaKimZhang2018uniformstab},
the empirical measure associated to any solution of the interacting particle system~\eqref{eq:cucker-smale} is itself a measure-valued solution to~\eqref{eq:mean_field_pde}.

The connection between the interacting particle system~\eqref{eq:cucker-smale} and the mean-field process~\eqref{eq:cucker-smale_mean_field} is not only of theoretical interest,
but can also be exploited for computational purposes.
Depending on the perspective,
the mean-field kinetic equation~\eqref{eq:mean_field_pde} can be used as an approximation of the interacting particle system~\eqref{eq:cucker-smale},
which is often the primary motivation for studying the mean-field limit~\cite{MR2425606},
or vice versa as in~\cite[Section 6]{MR2536440}.

\subsection{Literature review}
\label{sec:Literature_review}

\paragraph{Local-in-time propagation of chaos for the Cucker--Smale dynamics}
Even for a finite time interval,
proving propagation of chaos for~\eqref{eq:cucker-smale} requires bespoke techniques,
because the interaction is not globally Lipschitz continuous.
The first result in this direction was obtained in~\cite{MR2536440},
where a quantitative local-in-time mean-field result is obtained for a compactly supported initial law~$\mfldis_0$.
More precisely, the authors study the convergence with respect to~$h$ of a deterministic particle method for solving the nonlinear equation~\eqref{eq:mean_field_pde},
based on an approximation of the initial distribution~$\mfldis_0$ by an appropriate empirical measure over a lattice of size~$h$.

Shortly after,
\citeauthor{MR2860672}~\cite{MR2860672}
proved a quantitative local-in-time propagation of chaos estimate,
assuming chaotic initial data and strong control of the moments of the velocity marginal of~$\mfldis_0$.
The technique presented there is a rather general strategy,
summarized in~\cite[Section~3.1.2]{ReviewChaintronII},
which enables to prove quantitative propagation of chaos estimates for a wide class of models with only locally Lispchitz interactions,
at the cost of very strong moment bounds on the solution to the nonlinear Markov process associated with the interacting particle system.
The authors of~\cite{MR2860672} also prove global existence and uniqueness results for the interacting particle system~\eqref{eq:cucker-smale} and the nonlinear process~\eqref{eq:cucker-smale_mean_field},
again under strong moment bounds on the initial law~$\mfldis_0$.

The local-in-time mean-field limit for the Cucker--Smale model is revisited later in~\cite{MR4405302},
where the authors use an analytic rather than probabilistic coupling argument
(which however appears equivalent to a probabilistic argument à la Sznitman),
to obtain quantitative propagation of chaos estimates for a generalized Cucker--Smale dynamics,
assuming that~$\mfldis_0$ is compactly supported.

In the aforementioned references, the communication rate $\psi$ is assumed to satisfy local Lipschitz continuity assumptions.
For a proof of propagation of chaos for singular communication rates $\psi$ of the form $\psi(r):=r^{-\alpha}$ for~$\alpha>0$, see~\cite{HaKimPicklZhang2019probabilisticsingular},
which is based on the general methodology presented in~\cite{boers2016mean,MR3667287}.
In this reference, the authors prove a quantitative propagation of chaos estimate on finite time intervals,
in a metric akin to convergence in probability,
for an interacting particle system where the singularity in the communication rate is cut off at a scale depending on the number of particles.

\paragraph{Uniform-in-time mean-field limit for the Cucker--Smale dynamics}
The first uniform-in-time mean-field limit result for~\eqref{eq:cucker-smale} came from~\cite{HaKimZhang2018uniformstab}.
We will state this result precisely,
in order to contrast it with our main results at the end of~\cref{sec:main_results},
but before doing so we introduce the following assumption,
which is sufficient to guarantee exponential flocking for model~\eqref{eq:cucker-smale},
see~\cite[Section 3]{MR2536440}.
\begin{assumption}
    \label{assumption:cpt-support-and-K-large-enough}
    The probability distribution $\mfldis_0 \in \mathcal P(\R^{2d})$ has compact support and
    the communication strength $K$ is sufficiently large, that is,
    \begin{align}
        \label{eq:assump:K-large-enough}
        K > \frac{\diametermfl_V(0)}{\int_{\diametermfl_X(0)}^{\infty} \psi(s) \, \d s},
    \end{align}
    where the diameters of the support of~$\mfldis_0$ in position and velocity space are defined as
    \begin{equation}
        \label{eq:diametersmfl}
        \diametermfl_X(0) = \max_{x,y \in \supp_x (\mfldis_0)} \lvert x-y\rvert , \qquad
        \diametermfl_V(0) =\max_{v,w \in \supp_v (\mfldis_0)} \lvert v-w\rvert.
    \end{equation}
\end{assumption}
The first part of~\cite[Corollary~1]{HaKimZhang2018uniformstab} can then be summarized as follows:
\begin{theorem}
    \label{thm:ha-kim-zhang:mean-field-uit}
    (Approximation of the mean-field system by the particle system)
    Let \cref{assumption:cpt-support-and-K-large-enough} hold.
    Then, there exists a sequence of initial configurations $\bra*{\mathcal{Z}^n_0}_{n\in\N}$ where
    \begin{align}
        \mathcal{Z}^n_0 = \bra*{X^1_0, \dots, X^{J_n}_0, V^1_0, \dots, V^{J_n}_0} \in \R^{2d J_n} \quad \text{and } J_n\in\N,
    \end{align}
    such that the following holds: For each $n\in\N$, consider the solution to the interacting particle system~\eqref{eq:cucker-smale} with initial condition $\mathcal{Z}^n_0\in \R^{2dJ_n}$ and denote by $\emp{n}{t}$ the empirical measure at time $t\ge 0$ associated with this solution. Then, it holds that
    \begin{align}
        \lim_{n\to \infty} \sup_{t\in [0,\infty)} W_2(\emp{n}{t}, \mfldis_t) = 0.
    \end{align}
\end{theorem}

The proof of \cref{thm:ha-kim-zhang:mean-field-uit} relies crucially on the uniform-in-time stability estimate \cite[Theorem 1.2]{HaKimZhang2018uniformstab} for the interacting particle system~\eqref{eq:cucker-smale},
and it employs an argument to reinterpret solutions to~\eqref{eq:cucker-smale} with different system sizes~$J_1$ and~$J_2$
as solutions with the same number of particles~$J^+ \geq \max\{J_1, J_2\}$.
The number~$J^+$ can be taken as the least common multiple between~$J_1$ and~$J_2$.
This trick is possible because,
the evolution~\eqref{eq:cucker-smale} being deterministic,
the empirical measure associated with the solution at time~$t$ depends only on the empirical measure at time~0,
and not on the number or ordering of particles.
That is to say,
denoting by~$(\allz{J}_t)$ and~$(\allzt{nJ}_t)$ two solutions of~\eqref{eq:cucker-smale} with respectively~$J$ and~$nJ$ particles,
we have
\[
    \emp{J}{0} = \empt{nJ}{0} \quad \Rightarrow \quad \emp{J}{t} = \empt{nJ}{t}, \qquad \text{for all $t\ge0$}.
\]

The result~\cref{thm:ha-kim-zhang:mean-field-uit} is proved in the interaction regime \eqref{eq:assump:K-large-enough} where flocking for the interacting particle system~\eqref{eq:cucker-smale} is known to occur exponentially fast,
i.e.\ in a regime where the velocities of all agents in~\eqref{eq:cucker-smale} become arbitrarily close in time,
both in direction and magnitude.
The validity of a uniform-in-time mean-field estimate in the setting of exponential flocking is not surprising,
since the system practically stops moving after a finite time,
relative to its (usually time-dependent) space barycenter.
There are many references on flocking for the Cucker--Smale dynamics, we mention~\cite{MR2425606,MR2596552},
in addition to the references~\cite{MR2536440,HaKimZhang2018uniformstab} already mentioned.

\bigskip
Another approach for establishing uniform-in-time propagation of chaos for the Cucker-Smale dynamics is to use properties of gradient systems.
In the one-dimensional case $d=1$, the dynamics \eqref{eq:cucker-smale} can be written as a gradient system \cite{ha2018firstorderreduction, kim2022cuckersmalehessianfirstorderreduction}, for which~\cite{MR1847094} shows uniform-in-time propagation of chaos under appropriate assumptions on $\psi$.
For a variant of \eqref{eq:cucker-smale} with matrix-valued communication rates which reduces to \eqref{eq:cucker-smale} in the one-dimensional case, fibered gradient flows can be applied to show uniform-in-time propagation of chaos \cite{peszek2023heterogeneous, peszek2022measure}.

\paragraph{Recent developments on uniform-in-time propagation of chaos for various models}
To conclude this section,
let us mention a few of the milestones among the flurry of advances concerning uniform-in-time mean-field limits in recent years.
We refer to~\cite[Section 3]{ReviewChaintronII} for a more comprehensive overview.
In the already mentioned~\cite{MR1847094},
uniform-in-time propagation of chaos is proved for a gradient system in a convex potential via a synchronous coupling approach,
and this result is used to prove exponentially fast convergence of the associated nonlinear process. The idea is extended to a class of kinetic dynamics, which however does not include the Cucker--Smale model, in~\cite{MR2731396}.
Further extensions and generalizations were obtained in~\cite{MR3646428,MR4333408}.

In~\cite{MR4163850}, a new technique based on reflection couplings is employed to relax convexity assumptions in uniform-in-time propagation-of-chaos results for gradient systems,
and this technique is later adapted to the kinetic setting in~\cite{MR4489825}.
Let us also mention the work of Lacker~\cite{MR4595391},
later adapted in \cite{MR4634344} to the uniform-in-time setting,
where a novel approach based on the BBGKY hierarchy is deployed to obtain for a class of interacting particle systems that an appropriate distance
between the marginal law on the first~$k$ particles of the interacting system on one hand,
and the $k$-fold product of the mean-field law on the other hand,
scales as $\left(\frac{k}{N}\right)^2$,
where previous arguments would only give a scaling as $\frac{k}{N}$.

Let us finally mention the recent preprint~\cite{schuh2024conditionsuniformtimeconvergence},
where general conditions are presented under which uniform-in-time convergence can be deduced from local-in-time consistency estimates,
in a variety of settings including numerical discretizations of stochastic differential equations (SDEs), multiscale methods and mean-field systems.

\subsection{Contribution and organization of the paper}

The main results of this paper are the uniform-in-time mean-field estimates~\cref{theorem:uit_prop_of_empirical_chaos} and \cref{theorem:uit_prop_of_chaos},
which we present in \cref{sec:main_results}. After stating the auxiliary results in \cref{sec:auxiliary}, we present the proofs of the main results in~\cref{sec:Uniform-in-time_empirical_chaos} and~\cref{sec:Uniform-in-time_propagation_of_chaos}.
The proofs for the auxiliary results are presented in \cref{sec:aux-proofs}.
Finally, in~\cref{sec:numerics},
we illustrate numerically the main result \cref{theorem:uit_prop_of_chaos}.

\section{Main results}
\label{sec:main_results}

In this short paper we prove that,
based on a uniform-in-time stability estimate for the interacting particle system~\eqref{eq:cucker-smale}
and a decay estimate for the nonlinear process~\eqref{eq:cucker-smale_mean_field},
quantitative uniform-in-time propagation of chaos easily follows.
We present this approach for the particular case of the Cucker--Smale model,
but the strategy can be applied more generally.
It would be interesting, for example,
to investigate whether a similar analysis can be carried out for optimization methods based on interacting particle systems,
such as ensemble Kalman inversion~\cite{MR3041539,MR3988266,MR4521662} and recently proposed methods based on consensus formation~\cite{CBO,carrillo2018analytical,fornasier2024consensus}.
Local-in-time propagation of chaos estimates for these systems are proved in~\cite{MR4199469,MR4234152,vaes2024sharppropagationchaosensemble} and~\cite{CBO-mfl-Huang2021,gerber2023meanfield,koß2024meanfieldlimitconsensus}, respectively.

We emphasize that, in contrast with the classical reference~\cite{MR2536440} on the mean-field limit for the Cucker--Smale dynamics,
where initial configurations for~\eqref{eq:cucker-smale} supported on a regular lattice are considered in order to understand the convergence of the particle-in-cell method for solving the nonlinear equation~\eqref{eq:mean_field_pde},
we consider only the setting where the particles forming the initial configuration are randomly distributed in an independent, identically distributed~(i.i.d.) manner,
with common law~$\mfldis_0$.
This setting was studied in~\cite{MR2860672},
and the interacting particle system~\eqref{eq:cucker-smale} with this initial condition may be viewed as an alternative,
Monte Carlo-type approach to approximate the mean-field dynamics,
which is better suited to high-dimensional problems.

\paragraph{Preliminaries}
Before stating our main results,
we introduce some useful notation.
In order to prove mean-field limits which are uniform-in-time, we have to translate the empirical measure of the system \eqref{eq:cucker-smale} as well as the mean-field law $\mfldis_t$ from \eqref{eq:cucker-smale_mean_field} such that their barycenters coincides with the origin.
The reason is the following behavior of the models \eqref{eq:cucker-smale} and \eqref{eq:cucker-smale_mean_field} \cite[Section 2]{MR2536440}:
\[
    \frac{1}{J} \sum_{j=1}^{J} \xn{j}_t = \frac{1}{J} \sum_{j=1}^{J} \xn{j}_0 +  \left( \frac{1}{J} \sum_{j=1}^{J} \vn{j}_0 \right) t
    \qquad \text{and} \qquad \expect \bigl[\xl_t\bigr] = \expect \bigl[\xl_0\bigr] + \expect \bigl[\vl_0\bigr] t \qquad \text{for all $t\ge0$.}
\]
Therefore, if for instance \eqref{eq:cucker-smale} is started at two different initial configurations, then the two systems will drift away from each other unless their initial average velocities are the same.
To circumvent this difficulty, we introduce the recentering operator~$\mathcal{R} \colon \mathcal P_1(\R^{m}) \to \mathcal P_1(\R^{m})$ for $m\in\N^+$
to be the pushforward of~$\mu$ under the translation map~$z \mapsto z -  \int_{\R^m} w \, \mu(\d w)$, that is,
\begin{equation}
    \label{eq:recentering}
    \mathcal{R} \mu (A) = \mu \bra*{ A + \int_{\R^m} z \, \mu(\d z)} \qquad \text{for all Borel set $A \subset \R^{m}$.}
\end{equation}
We also need a recentering operator  $\mathcal R^J$ for the joint law of the particles~\eqref{eq:cucker-smale}.
The operator $\mathcal R^J:\mathcal P(\R^{2dJ})\to\mathcal P(\R^{2dJ}):$ is defined as the pushforward under the map~$R^J\colon \R^{2dJ} \to \R^{2dJ}$ given by
\[
    R^J(\mathbf Z) = \left( \zn{1} - \frac{1}{J} \sum_{j=1}^{J} \zn{j}, \dotsc , \zn{J} - \frac{1}{J} \sum_{j=1}^{J} \zn{j} \right).
\]
This operator is different from the recentering operator~\eqref{eq:recentering},
as $R^J$ acts at the level of configurations in $\R^{2dJ}$,
but the two operators are related.
To make this link explicit,
consider the empirical measure map~$m \colon \real^{2dJ} \to \mathcal P(\real^{2d})$ which,
to a vector~$\mathbf Z \in \real^{2dJ}$, associates the empirical measure
\(
m(\mathbf Z) \coloneq \frac{1}{J} \sum_{j=1}^{J} \delta_{\zn{j}}.
\)
Then for any probability measure~$\rho \in \mathcal P(\R^{2dJ})$,
it follows from the equality $ m \circ R^J=\mathcal R \circ m$ that $m_\sharp \bigl(\mathcal R^J(\rho)\bigr) = \mathcal R_{\sharp} \bigl(m_\sharp (\rho)\bigr)$.

Next, we recall two different notions of propagation of chaos considered in this paper
(see \cite{ReviewChaintronI} for a more thorough discussion of different versions of propagation of chaos).
\begin{itemize}
    \item
          \emph{Empirical chaos} refers to the property that the empirical measure associated with the solution to the interacting particle system~\eqref{eq:cucker-smale},
          a random probability measure on~$\real^{2d}$,
          converges to the deterministic mean-field law~$\mfldis_t$ in the limit as~$J \to \infty$.

    \item
          \emph{Infinite dimensional Wasserstein chaos} refers to the property that,
          in the limit~$J \to \infty$,
          the probability distribution~$\rho^J_t \in \mathcal P(\real^{2dJ})$ becomes closer (in an appropriately normalized Wasserstein distance)
          to the product measure~$\mfldis_t^{\otimes J} \in \mathcal P(\real^{2dJ})$,
          where~$\rho^J_t$ is the law of the solution to the interacting particle system~\eqref{eq:cucker-smale} at time~$t$.
\end{itemize}
The two notions of chaos are related but not equivalent in general,
with infinite dimensional Wasserstein-2 chaos being the stronger property;
see for example~\cite[Lemma 4.2]{ReviewChaintronI}.
In this paper we shall first prove empirical chaos uniformly in time,
and only then focus on infinite dimensional Wasserstein-2 chaos uniformly in time,
the reason being that empirical chaos is particularly simple to prove for~\eqref{eq:cucker-smale} given existing results in the literature.

Having introduced the recentering operators and recalled different notion of chaos,
we are now in a position to state our main results.
Our first main result is the following:
\begin{proposition}
    [Uniform-in-time empirical chaos, with a rate]
    \label{theorem:uit_prop_of_empirical_chaos}
    Let \cref{assumption:cpt-support-and-K-large-enough} be satisfied
    and assume that the initial configurations $\bigl(\xinitn{j}, \vinitn{j}\bigr)_{j \in \N}$ are sampled i.i.d.\ from $\mfldis_0$.
    There exists a constant $C > 0$ such that for all~$J \in \N^+$, we have
    \begin{equation}
        \label{eq:empirical_chaos_statement}
        \sup_{t \geq 0} \left( \expect  W_2\bigl(\mathcal R \emp{J}{t}, \mathcal R \mfldis_t\bigr)^2 \right)^{\frac{1}{2}}
        \leq C
        \begin{cases}
            J^{-\frac{1}{4}}                           & \text{if $d < 4$} \, , \\
            J^{-\frac{1}{4}} \log(1 + J)^{\frac{1}{2}} & \text{if $d = 4$} \, , \\
            J^{-\frac{1}{d}}                           & \text{if $d > 4$} \, ,
        \end{cases}
    \end{equation}
    where $\emp{J}{t}$ is the empirical measure associated with the solution to the interacting particle system~\eqref{eq:cucker-smale} at time~$t$ and~$\mfldis_t$ is the solution to the nonlinear process~\eqref{eq:cucker-smale_mean_field}.
\end{proposition}
The proof of \cref{theorem:uit_prop_of_empirical_chaos} is based on the uniform-in-time stability estimate from \cite{HaKimZhang2018uniformstab}, a simple particle duplication trick as well as \cite[Theorem 1]{MR3383341}.

With additional work,
we can improve on \cref{theorem:uit_prop_of_empirical_chaos} to show infinite dimensional Wasserstein-2 propagation of chaos uniformly in time.
Assume again the initial condition for \eqref{eq:cucker-smale} is sampled i.i.d.\ from $\mfldis_0$. Following the classical synchronous coupling approach \cite{Sznitman,ReviewChaintronI},
we introduce the following system composed of~$J$ copies of the mean-field dynamics~\eqref{eq:cucker-smale_mean_field},
with the same i.i.d.\ initial condition as~\eqref{eq:cucker-smale}:
\begin{equation}
    \label{eq:cucker-smale_mean_field_synch}
    \left\{
    \begin{aligned}
        \d \xnl{j}_t & = \vnl{j}_t \, \d t                                                                 \\
        \d \vnl{j}_t & = -\left(H \star \mfldis_t\right) \left(\xnl{j}_t, \vnl{j}_t\right) \, \d t, \qquad
        \qquad \left(\xnl{j}_0, \vnl{j}_0\right) =
        \left(\xn{j}_0, \vn{j}_0\right) =
        \left(\xinitn{j}, \vinitn{j}
        \right)\overset{\text{i.i.d.}}{\sim}\mfldis_0,
        \\
        \mfldis_t    & = {\rm Law}(\xl_t, \vl_t).
    \end{aligned}
    \right.
\end{equation}
Note that the processes $\bigl(\xnl{j}_t, \vnl{j}_t\bigr)_{t\geq 0}$ are independent and identically distributed. We introduce the notation
\[
    \Delta \xn{j}_t = \xn{j}_t - \frac{1}{J} \sum_{j=1}^{J} \xn{j}_t,  \qquad
    \Delta \vn{j}_t = \vn{j}_t - \frac{1}{J} \sum_{j=1}^{J} \vn{j}_t, \qquad
    \Delta \xnl{j}_t = \xnl{j}_t - \frac{1}{J} \sum_{j=1}^{J} \xnl{j}_t,  \qquad
    \Delta \vnl{j}_t = \vnl{j}_t - \frac{1}{J} \sum_{j=1}^{J} \vnl{j}_t.
\]
The second main result and main contribution is then as follows:

\begin{theorem}
    [Uniform-in-time propagation of chaos, with a rate]
    \label{theorem:uit_prop_of_chaos}
    Suppose that \cref{assumption:cpt-support-and-K-large-enough} holds.
    Then there exists a constant $C_{\rm Chaos} =  C_{\rm Chaos}\bigl(K, \diametermfl_V(0), \diametermfl_X(0), \lip \bigr) > 0$ such that for all $J \in \N^+$,
    \begin{equation}
        \label{eq:chaos_main}
        \forall t \geq 0, \qquad
        \expect \left[ \bigl\lvert \Delta \xn{j}_t - \Delta \xnl{j}_t \bigr\rvert^2 + \bigl\lvert \Delta \vn{j}_t - \Delta \vnl{j}_t \bigr\rvert^2 \right]
        \leq \frac{C_{\rm Chaos}^2}{J},
    \end{equation}
    Consequently, it holds for all $t \geq 0$ that
    \(
    W_2 \bigl(\mathcal R^J \rho^J_t, \mathcal R^J \mfldis_t^{\otimes J} \bigr) \leq C_{\rm Chaos} J^{-\frac{1}{2}} \, ,
    \)
    where $W_2$ denotes the Wasserstein distance on $\real^{2dJ}$ with respect to the following normalized distance as in~\cite[Definition 3.5]{ReviewChaintronI} :
    \begin{align}
        d(\mathrm{x}, \mathrm{y}) := \frac{1}{J}\bra[\Big]{\sum_{j=1}^{J} \bigl\lvert \mathrm{x}_j - \mathrm{y}_j \bigr\rvert^2}^{\frac{1}{2}}
        \qquad \text{for } \mathrm{x} = (\mathrm{x}_1, \ldots, \mathrm{x}_J) \in \real^{2dJ}, \quad
        \mathrm{y} = (\mathrm{y}_1, \ldots, \mathrm{y}_J) \in \real^{2dJ} \, .
    \end{align}
\end{theorem}
The proof of \cref{theorem:uit_prop_of_chaos} follows in spirit the ubiquitous idea in numerical analysis that consistency (in our case~\cref{theorem:local_mean_field}) and stability (in our case \cref{theorem:stability_ips}) together imply convergence (\cref{theorem:uit_prop_of_chaos}).
For other applications of this approach,
see for example Lax and Richtmyer's paper on the convergence of finite difference schemes~\cite[Section 8]{MR79204},
weak convergence estimates for numerical schemes for stochastic differential equations (SDEs)~\cite[Section 7.5.2]{MR3097957},
Trotter--Kato approximation theorems~\cite[Chapter~4]{MR2229872},
and uniform-in-time averaging results for multiscale~SDEs~\cite{crisan2024poissonequationslocallylipschitzcoefficients},
to mention just a few.
We refer again to~\cite{schuh2024conditionsuniformtimeconvergence} for unifying conditions to prove uniform-in-time convergence.

While the simple argument to prove~\cref{theorem:uit_prop_of_empirical_chaos} works only for interacting particle system which,
apart from the initial condition,
are deterministic,
the method of proof for~\cref{theorem:uit_prop_of_chaos} presented in~\cref{sec:Uniform-in-time_propagation_of_chaos}
does not rely on the deterministic nature of the evolution~\eqref{eq:cucker-smale} and can be applied to any interacting particle system for which a suitable uniform-in-time stability estimate and a suitable decay estimate for the mean-field process are available.

To conclude this section,
we provide a general comparison of our main results with \cite{HaKimZhang2018uniformstab} in~\cref{rmk:comparison-with-ha-kim-zhang}
and discuss an alternative approach to recentering in~\cref{rmk:alternative-cucker-smale-system}.
\begin{remark}
    [Comparison of this work with existing literature]
    \label{rmk:comparison-with-ha-kim-zhang}
    Our main results (\cref{theorem:uit_prop_of_empirical_chaos,theorem:uit_prop_of_chaos}) differ from~\cref{thm:ha-kim-zhang:mean-field-uit} in mainly two ways:
    \begin{itemize}
        \item On the one hand, in this work we consider mean-field limits from the viewpoint of propagation of chaos. Instead of finding a sequence of deterministic initial particle configurations such that the corresponding sequence of solutions of~\eqref{eq:cucker-smale} converges to the solution to the mean-field equation \eqref{eq:mean_field_pde}, we assume that the initial configurations are sampled i.i.d.\ from the mean-field law $\mfldis_0$ and prove empirical and infinite dimensional Wasserstein-2 propagation of chaos.

        \item On the other hand, our main results \cref{theorem:uit_prop_of_empirical_chaos,theorem:uit_prop_of_chaos} provide quantitative error bounds, while \cref{thm:ha-kim-zhang:mean-field-uit} is not quantitative.
              We provide a quantitative version of~\cref{thm:ha-kim-zhang:mean-field-uit} in \cref{theorem:ha-kim-zhang:mean-field-uit-quantitative}.
    \end{itemize}
\end{remark}

\begin{remark}
    [Alternative Cucker--Smale system]
    \label{rmk:alternative-cucker-smale-system}
    Instead of recentering through the maps $\mathcal{R}$ and $\mathcal{R}^J$,
    we could also have worked with the modified particle system
    \begin{align}
        \label{eq:cucker-smale-without-drift}
        \frac{\d}{\d t} \xnt{j}_t & = \vnt{j}_t - \frac{1}{J}\sum_{i=1}^J \vnt{i}_t,
        \qquad                    &
        \frac{\d}{\d t} \vnt{j}_t & = -\frac{K}{J} \sum_{k=1}^{J}
        \psi \Bigl( \bigl\lvert \xnt{j}_t - \xnt{k}_t \bigr\rvert \Bigr)
        \left( \vnt{j}_t - \vnt{k}_t \right)
    \end{align}
    and a similar modification of the mean-field equation~\eqref{eq:cucker-smale_mean_field}.
    If \eqref{eq:cucker-smale} and \eqref{eq:cucker-smale-without-drift} start at the same initial conditions
    \begin{align}
        \xn{j}_0 = \xnt{j}_0, \quad \text{and} \quad \vn{j}_0=\vnt{j}_0 \qquad \text{for all $j\in\range{1}{J}$},
    \end{align}
    then the two systems are related by
    \begin{align}
        \xn{j}_t = \xnt{j}_t + \frac{1}{J}\sum_{i=1}^J \vn{i}_0\quad\text{and}\quad
        \vn{j}_t = \vnt{j}_t \qquad \text{for all $t\ge 0$ and $j\in\range{1}{J}$}.
    \end{align}
    Therefore,
    an alternative and equivalent approach would have been to prove mean-field limits for \eqref{eq:cucker-smale-without-drift} and then to translate these results to mean-field limits for \eqref{eq:cucker-smale}.
\end{remark}

\section{Auxiliary results}
\label{sec:auxiliary}
Our main results are based on the following three auxiliary results.

\begin{theorem}
    [Finite-time propagation of chaos]
    \label{theorem:local_mean_field}
    Assume that $\mfldis_0 \in \mathcal P(\R^{2d})$ is compactly supported in velocities.
    Then it holds that
    \begin{equation}
        \label{eq:local_mean_field}
        \forall J \in \N^+, \qquad \forall t \in [0, \infty), \qquad
        \expect \left[ \bigl\lvert \xn{j}_t - \xnl{j}_t \bigr\rvert^2 + \bigl\lvert \vn{j}_t - \vnl{j}_t \bigr\rvert^2 \right]
        \leq
        \frac{ 2 \bra*{C_{\rm MF}}^t}{J}
        \, \expect \left\lvert \vnl{1}_0 - \expect \left[ \vnl{1}_0 \right]  \right\rvert^2.
    \end{equation}
    Here
    \(
    C_{\rm MF} = \exp \Bigl(1 + 2 K \lip \diametermfl_{V}(0) + K\Bigr),
    \) where $\diametermfl_{V}(0)$ is the diameter of the velocity marginal of~$\mfldis_0$, see \eqref{eq:diametersmfl}.
\end{theorem}

\cref{theorem:local_mean_field} is an adaptation of the local-in-time propagation of chaos result~\cite[Theorem 1.1]{MR2860672},
where we make explicit the dependence of the prefactor on the initial probability distribution~$\mfldis_0$
and consider initial conditions that are compactly supported in velocities.
The proof can be found in~\cref{sec:proof-local-mean-field}.
Observe that $\expect \bigl[ \lvert \xn{j}_t - \xnl{j}_t \rvert^2 \bigr]$ is bounded independently of $\expect \bigl\lvert \xnl{1}_0 - \expect \xnl{1}_0 \bigr\rvert^2$.
This may be intuitively understood by noting that
\[
    \xn{j}_t - \xnl{j}_t = \int_{0}^{t} (\vn{j}_s - \vnl{j}_s) \, \d s \, ,
\]
and so it is not surprising that closeness of velocities leads to closeness of positions.

\begin{theorem}
    [Uniform-in-time stability estimate]
    \label{theorem:stability_ips}
    For fixed $K>0$,
    there exists a function~$C_{\rm Stab}\colon [0,\infty)^3 \to [0,\infty)$ which is non-decreasing in each of its arguments, such that the following holds.
    Let~$J \in \N$ and let $(\xn{j}_t, \vn{j}_t)_{j \in \range{1}{J}}$ and~$(\xnt{j}_t, \vnt{j}_t)_{j \in \range{1}{J}}$
    denote two solutions to the interacting particle system~\eqref{eq:cucker-smale} with initial positions
    and velocities~$\xn{j}_0, \vn{j}_0 \in \R^d$ and $\xnt{j}_0,\vnt{j}_0 \in \R^d$ for~$j\in\range{1}{J}$.
    Additionally, suppose that
    \begin{align}
        \label{eq:assump_K_stab}
        \sum_{j=1}^{J} \vn{j}_0 = \sum_{j=1}^{J} \vnt{j}_0 = 0 \qquad \text{and} \qquad
        K > \max \left\{ \frac{\diameter_V(0)}{\int_{\diameter_X(0)}^{\infty} \psi(s) \, \d s}, \frac{\diameter_{\widetilde V}(0)}{\int_{\diameter_{\widetilde X}(0)}^{\infty} \psi(s) \, \d s} \right\},
    \end{align}
    where we use the notation
    \begin{equation}
        \label{eq:diameters1}
        \diameter_X(t) = \max_{1\le j,k\le J} \lvert \xn{j}_t - \xn{k}_t \rvert, \qquad
        \diameter_V(t) = \max_{1\le j,k\le J} \lvert \vn{j}_t - \vn{k}_t \rvert
    \end{equation}
    and a similar one for $\diameter_{\widetilde X}(0)$ and $\diameter_{\widetilde V}(0)$.
    Then it holds that
    \begin{align}
        \label{eq:uniform_stab_statement}
        \sup_{t \geq 0}
        \bra*{
            \sum_{j=1}^{J} \bigl\lvert \xn{j}_s - \xnt{j}_s \bigr\rvert^2
            + \sum_{j=1}^{J} \bigl\lvert \vn{j}_s - \vnt{j}_s \bigr\rvert^2}^{\frac{1}{2}}
        \leq
        C_{\rm Stab}
        \bra*{
            \sum_{j=1}^{J} \bigl\lvert \xn{j}_0 - \xnt{j}_0 \bigr\rvert^2
            + \sum_{j=1}^{J} \bigl\lvert \vn{j}_0 - \vnt{j}_0 \bigr\rvert^2
        }^{\frac{1}{2}},
    \end{align}
    where the stability constant is given by $C_{\rm Stab}:=C_{\rm Stab}\bigl(\diameter_X(0), \diameter_V(0), \diametert_V(0)\bigr)$.
\end{theorem}
\cref{theorem:stability_ips} was proved in \cite[Theorem 1.2]{HaKimZhang2018uniformstab}.
For the reader's convenience,
we provide a mostly self-contained proof in \cref{sec:proof-stability}.

\begin{lemma}
    [Exponential concentration in velocity]
    \label{theorem:exponential_concenttration}
    Under \cref{assumption:cpt-support-and-K-large-enough},
    there exist constants~$C_{\rm Decay}<\infty$ and~$\lambda>0$ depending on $\diametermfl_X(0), \diametermfl_V(0), \psi$ and $K$ such that
    \begin{equation}
        \forall t \geq 0, \qquad
        \expect \Bigl\lvert \vl_t - \expect \bigl[\vl_t\bigr] \Bigr\rvert^2
        \leq
        C_{\rm Decay} \e^{- \lambda t} \expect \Bigl\lvert \vl_0 - \expect \bigl[\vl_0\bigr] \Bigr\rvert^2.
    \end{equation}
\end{lemma}
Estimates of this form, also called flocking estimates, can be derived from
similar estimates for the interacting particles system~\eqref{eq:cucker-smale}
via local-in-time propagation of chaos as in~\cref{theorem:local_mean_field} and a limiting procedure.
We refer to~\cite[Corollary 2]{HaKimZhang2018uniformstab} for a full proof relying on this strategy.
Alternatively, they can also be obtained by direct study of the nonlinear PDE associated with~\eqref{eq:cucker-smale_mean_field},
as in~\cite[Proposition 5]{MR2596552} and~\cite[Theorem 3]{MR3432849}.

\section{\texorpdfstring{Proof of \cref{theorem:uit_prop_of_empirical_chaos}:}{} Uniform-in-time empirical propagation of chaos}
\label{sec:Uniform-in-time_empirical_chaos}

The main aim of this section is to prove~\cref{theorem:uit_prop_of_empirical_chaos}.
To this end, we first use \cref{theorem:stability_ips} to prove a stability estimate for the empirical measures in Wasserstein distance, see \cref{sec:cor-comparing-empirical-laws-different-J}.
Using this stability result,
we prove~\cref{theorem:uit_prop_of_empirical_chaos} in~\cref{sec:uit_empirical}.
Finally, in \cref{sec:quantitative-version-of-has-result},
we state a quantitative version of \cref{thm:ha-kim-zhang:mean-field-uit}.

\subsection{Empirical uniform-in-time stability in Wasserstein distance}
\label{sec:cor-comparing-empirical-laws-different-J}

Using a particle duplication trick, we first prove the following consequence of the uniform-in-time stability estimate~\cref{theorem:stability_ips}.
A similar result,
also based on a particle duplication trick,
was obtained as an intermediate step in the proof of~\cite[Corollary~1]{HaKimZhang2018uniformstab}.
We give a self-contained proof below,
both for completeness and because
our approach is somewhat more direct,
as it avoids the construction of a rational approximation for an optimal transference plan.
\begin{corollary}
    [Wasserstein stability for the empirical measures]
    \label{cor:wasserstein-stability-empirical-measures}
    Consider $J, \tilde{J} \in \N^+$ together with initial conditions~$ \allz{J}_0= \bigl(\allx{J}_0, \allv{J}_0\bigr) \in \R^{2dJ}$ and $\allzt{\tilde{J}}_0= \bigl(\allxt{\tilde{J}}_0, \allvt{\tilde{J}}_0\bigr) \in \R^{2d\tilde{J}}$ such that~\eqref{eq:assump_K_stab} holds.
    Then, for the corresponding solutions to the interacting particle system~\eqref{eq:cucker-smale},
    it holds that
    \begin{equation}
        \label{eq:uniform_stability_in_statement}
        W_2 \Bigl( \emp{J}{t}, \empt{\widetilde J}{t} \Bigr)
        \leq
        C_{\rm Stab}
        W_2 \Bigl( \emp{J}{0}, \empt{\widetilde J}{0} \Bigr)
        \qquad \allz{J}_t := \Bigl(\allx{J}_t, \allv{J}_t\Bigr), \qquad
        \allzt{J}_t := \Bigl(\allxt{\widetilde J}_t, \allvt{\widetilde J}_t\Bigr).
    \end{equation}
\end{corollary}

\begin{proof}
    First note that, for two solutions $\bigl(\allz{J}_t\bigr)$ and $\bigl(\allzt{J}_t\bigr)$
    with the same number of particles,
    it holds~\cite[p.5]{MR1964483} that
    the Wasserstein distance between empirical measures the associated empirical measures
    is equal to
    \begin{equation}
        \label{eq:permutation_wasserstein}
        W_2\bigl(\mu_{\allz{J}_t}, \mu_{\allzt{J}_t}\bigr)
        = \min_{\sigma \in \mathcal S_J} \left(\frac{1}{J} \sum_{j=1}^{J} \left\lvert \zn{j}_t - \znt{\sigma(j)}_t \right\rvert^2\right)^{\frac{1}{2}},
    \end{equation}
    where $\mathcal S_J$ denotes the set of permutations in $\{1, \dotsc, J\}$.
    Therefore, numbering the particles of each configuration in such a manner that,
    for~$t = 0$, equality in~\eqref{eq:permutation_wasserstein} is achieved for the identity permutation,
    it follows from~\eqref{eq:uniform_stab_statement} that
    \begin{align}
        \label{eq:wasserstein_stab_same_J}
        \forall t \geq 0, \qquad
        W_2 \bigl(\mu_{\allz{J}_t}, \mu_{\allzt{J}_t} \bigr)
        \leq
        \left( \sum_{j=1}^{J} \left\lvert \zn{j}_t - \znt{j}_t \right\rvert^2 \right)^{\frac{1}{2}}
        \leq
        C_{\rm Stab}
        \left( \sum_{j=1}^{J} \left\lvert \zn{j}_0 - \znt{j}_0 \right\rvert^2 \right)^{\frac{1}{2}}
        =
        C_{\rm Stab} W_2 \bigl(\mu_{\allz{J}_0}, \mu_{\allzt{J}_0} \bigr).
    \end{align}
    This concludes the proof for the case $J = \widetilde J$.
    To address the case where $J \neq \widetilde J$, we introduce some notation:
    given a configuration $\allz{J} = \left(\zn{1}, \ldots, \zn{J} \right) \in \R^{dJ}$,
    we denote by $\Psi_N(\allz{J}) \in \R^{dJN}$ the configuration obtained by cloning~$N-1$ times each particle of $\allz{J}$,
    that is to say
    \[
        \Psi_N(\allz{J}) = \bigl( \underbrace{\zn{1}, \ldots, \zn{1}}_{\times N}, \underbrace{\zn{2}, \ldots, \zn{2}}_{\times N}, \ldots, \underbrace{\zn{J}, \ldots, \zn{J}}_{\times N} \bigr).
    \]
    Note that, for all~$N \in \N^+$,
    the empirical measure~$\bigl(\emp{J}{t}\bigr)$ associated with the solution to~\eqref{eq:cucker-smale} with size~$J$
    coincides with the empirical measure~$\bigl(\empt{NJ}{t}\bigr)$ associated with the solution to the same system with size~$N J$,
    provided that wherever a particle of the small system is initialized,
    $N$ particles of the large system are initialized,
    i.e.~if $\allzt{NJ}_0 = \Psi_N(\allz{J}_0)$ up to permutations.
    Indeed, in this case the empirical measures coincide at the initial time,
    so both of them must equal the unique measure-valued solutions to the PDE~\eqref{eq:mean_field_pde};
    see~\cite[Remark 8]{HaKimZhang2018uniformstab}.
    Denoting by~$\mathcal E^K_t \colon \R^{2dK} \to \R^{2dK}$ the operator which maps initial conditions of~\eqref{eq:cucker-smale} with size~$K$ to the solution at time~$t$,
    we reformulate this as
    \begin{equation}
        \label{eq:propagation}
        \emp{J}{t} = \mu_{\mathcal E^{NJ}_t\left(\Psi_N\left(\allz{J}_0\right)\right)}.
    \end{equation}
    For $J,\widetilde J\in\N$,
    we consider now two initial configurations of particles $\allz{J}_0$ and $\allzt{\widetilde J}_0$ such that
    the means of the initial velocity vectors are zero, i.e.\ $\sum_{j=1}^J \vn{j}_0 = \sum_{j=1}^{\widetilde J} \vnt{j}_0 = 0$.
    Let $N = J \widetilde J$,
    and let $\allzf{N}_t, \allzft{N}_t \in \R^{2d J \widetilde J}$ denote the solutions to~\eqref{eq:cucker-smale} with initial conditions~$\Psi_{\widetilde J} (\allz{J}_0)$ and $\Psi_{J} (\allzt{\widetilde J}_0)$,
    respectively.
    By~\eqref{eq:propagation} and
    the uniform-in-time stability estimate~\eqref{eq:wasserstein_stab_same_J} for systems of the same size,
    it follows that
    \begin{align}
        W_2 \bigl(\mu_{\allz{J}_t}, \mu_{\allzt{\widetilde J}_t} \bigr)
         & =  W_2 \bigl(\mu_{\mathfrak Z^{N}_t}, \mu_{\widetilde{\mathfrak Z}^{N}_t}\bigr)                \\
         & \le C_{\rm Stab}  W_2 \bigl(\mu_{\mathfrak Z^{N}_0}, \mu_{\widetilde{\mathfrak Z}^{N}_0}\bigr)
        = C_{\rm Stab}  W_2 \bigl(\mu_{\allz{J}_0}, \mu_{\allzt{\widetilde J}_0} \bigr),
    \end{align}
    which concludes the proof.
\end{proof}

\begin{remark}
    The method used in this proof would not work for stochastic interacting particle systems.
\end{remark}

\subsection{Proof of \texorpdfstring{\cref{theorem:uit_prop_of_empirical_chaos}}{Proposition~\ref*{theorem:uit_prop_of_empirical_chaos}}}
\label{sec:uit_empirical}
In this section, we show how local-in-time empirical chaos
\cref{theorem:local_mean_field} together with a uniform stability estimate for the interacting particle system as in \cref{theorem:stability_ips} directly yields a uniform-in-time empirical chaos estimate.

\begin{proof}[Proof of \cref{theorem:uit_prop_of_empirical_chaos}]
    Fix $t > 0$,
    and first note that
    \(
    (\mathcal R \mfldis_t)_{t \geq 0}
    \)
    is the law of the solution to the nonlinear equation~\eqref{eq:cucker-smale_mean_field} with initial condition~$\mathcal R \mfldis_0$ instead of~$\mfldis_0$,
    and that for every realization of the initial data,
    \(
    (\mathcal R \emp{J}{t})_{t \geq 0}
    \)
    is the empirical measure associated to~\eqref{eq:cucker-smale} with recentered initial data.
    The local-in-time propagation of chaos result~\cref{theorem:local_mean_field} implies that,
    for all $t \geq 0$,
    \begin{equation}
        \lim_{J \to \infty} \expect \Bigl[ W_2 \bigl(\emp{J}{t}, \mfldis_t \bigr)^2 \Bigr] = 0.
    \end{equation}
    As a consequence, we deduce that
    \begin{equation}
        \label{eq:empirical_chaos_recentered}
        \forall  t \geq 0, \qquad
        \lim_{J \to \infty} \expect \Bigl[ W_2 \bigl(\mathcal R \emp{J}{t}, \mathcal R \mfldis_t \bigr)^2 \Bigr] = 0.
    \end{equation}
    Indeed, for any coupling $\pi$ between probability measures $\mu, \nu \in \mathcal P_1(\R^n)$,
    one may naturally define a coupling between the recentered measures $\mathcal R \mu$ and~$\mathcal R \nu$ as follows:
    for a Borel set $\mathcal A \in \mathcal B(\R^{2n})$, we set
    \[
        \widetilde \pi(A) := \pi \bigl(A + \lambda \bigr), \qquad  \lambda = \begin{pmatrix}
            \int_{\R^m} z \, \mu(\d z) \\ \int_{\R^m} z \, \nu(\d z) \end{pmatrix},
    \]
    From this coupling one easily obtains via a triangle inequality that
    \begin{align}
        \label{eq:W2-recentered-inequality}
        W_2(\mathcal R \mu, \mathcal R \nu) \leq W_2(\mu, \nu) +
        \abs*{ \int_{\R^m} z \, \mu(\d z) - \int_{\R^m} z \, \nu(\d z) }
        \leq 2 W_2(\mu, \nu).
    \end{align}
    Therefore, \eqref{eq:empirical_chaos_recentered} and \cref{cor:wasserstein-stability-empirical-measures} imply that for all~$J \in \N^+$ and all~$t \geq 0$ that
    \begin{align}
        \label{eq:triangle}
        \Bigl( \expect  W_2 \bigl(\mathcal R \emp{J}{t}, \mathcal R \mfldis_t \bigr)^2 \Bigr)^{\frac{1}{2}}
         & = \lim_{n\to +\infty} \Bigl( \expect  W_2 \bigl(\mathcal R \emp{J}{t}, \mathcal R \mu_{\allzt{nJ}_t} \bigr)^2 \Bigr)^{\frac{1}{2}}                 \\
         & \leq C_{\rm Stab} \lim_{n\to +\infty} \Bigl( \expect  W_2 \bigl(\mathcal R \emp{J}{0}, \mathcal R \mu_{\allzt{nJ}_0} \bigr)^2 \Bigr)^{\frac{1}{2}}
        = C_{\rm Stab} \Bigl( \expect  W_2 \bigl(\mathcal R \emp{J}{0}, \mathcal R \mfldis_0 \bigr)^2 \Bigr)^{\frac{1}{2}},
    \end{align}
    where $\mu_{\allzt{nJ}}$ denotes the empirical measure associated to the Cucker--Smale dynamics with~$nJ$ particles,
    initialized randomly and independently according to~$\mfldis_0$.
    From~\eqref{eq:W2-recentered-inequality} we have
    \begin{align}
        \expect \Bigl[ W_2 \Bigl(\mathcal R \emp{J}{0}, \mathcal R \mfldis_0 \Bigr)^2 \Bigr]
        \le 2 \expect \Bigl[ W_2 \bigl(\emp{J}{0}, \mfldis_0 \bigr)^2 \Bigr] .
    \end{align}
    On the other hand, by \cite[Theorem 1]{MR3383341}, there exists a constant $c>0$ such that
    \begin{align}
        \expect \Bigl[ W_2 \bigl(\emp{J}{0}, \mfldis_0 \bigr)^2 \Bigr]
        \le c \begin{cases}
                  J^{-\frac{1}{2}}             & \text{if $d < 4$} \, , \\
                  J^{-\frac{1}{2}} \log(1 + J) & \text{if $d = 4$} \, , \\
                  J^{-\frac{2}{d}}             & \text{if $d > 4$} \, .
              \end{cases}
    \end{align}
    Then claim follows then from \eqref{eq:triangle}.
\end{proof}

\subsection{A quantitative version of \texorpdfstring{\cref{thm:ha-kim-zhang:mean-field-uit}}{Theorem~\ref*{thm:ha-kim-zhang:mean-field-uit}}}
\label{sec:quantitative-version-of-has-result}
With a similar argument as in the proof of \cref{theorem:uit_prop_of_empirical_chaos}, we have the following quantitative version of~\cref{thm:ha-kim-zhang:mean-field-uit}, with a simplified proof compared to that of \cite[Corollary~1]{HaKimZhang2018uniformstab}.
\begin{proposition}
    \label{theorem:ha-kim-zhang:mean-field-uit-quantitative}
    Assume that \cref{assumption:cpt-support-and-K-large-enough} holds.
    Let a sequence of initial configurations $\bra*{\mathcal{Z}^n_0}_n$ be given, where
    \begin{align}
        \mathcal{Z}^n_0 = \bra*{X^1_0, \dots, X^{J_n}_0, V^1_0, \dots, V^{J_n}_0} \in \R^{2d J_n} \quad \text{and } J_n\in\N,
    \end{align}
    and suppose that
    \begin{align}
        \label{eq:approximating-intial-mfl-measure}
        \lim_{n\to\infty} W_2\bigl(\emp{n}{0}, \mfldis_0\bigr) =0.
    \end{align}
    For each $n\in\N$, consider the solution to the interacting particle system~\eqref{eq:cucker-smale} with initial condition $\mathcal{Z}^n_0$ and denote by~$\emp{n}{t}$ the empirical measure at time $t\ge 0$ associated with this solution. Then, it holds for all $n\in\N$ that
    \begin{align}
        \sup_{t\in [0,\infty)} W_2(\mathcal{R}\emp{n}{t}, \mathcal{R}\mfldis_t)  \le C_{\rm Stab} \cdot W_2\bigl(\mathcal{R}\emp{n}{0}, \mathcal{R}\mfldis_0\bigr)
        \le 2 C_{\rm Stab} \cdot W_2\bigl(\emp{n}{0}, \mfldis_0\bigr).
    \end{align}
\end{proposition}

\begin{proof}
    By \cref{cor:wasserstein-stability-empirical-measures}, we have for all $n,m\in\N$ and $t\in\N$ that
    \begin{align}
        W_2\bigl(\mathcal R\emp{n}{t}, \mathcal R\emp{m}{t}\bigr)
         & \le C_{\rm Stab} W_2\bigl({\mathcal R}\emp{n}{0}, {\mathcal R}\emp{m}{0}\bigr).
    \end{align}
    Letting $m\to\infty$, the claim follows from \eqref{eq:approximating-intial-mfl-measure}, the triangle inequality and \eqref{eq:W2-recentered-inequality}.
\end{proof}
\section{\texorpdfstring{Proof of \cref{theorem:uit_prop_of_chaos}:}{} Uniform-in-time infinite dimensional Wasserstein chaos}
\label{sec:Uniform-in-time_propagation_of_chaos}

The strategy for the proof of \cref{theorem:uit_prop_of_chaos}
can be described informally as follows.
Fix a probability measure~$\mfldis_0$ and denote by $(\mfldis_t)_{t\geq 0}$ the associated mean-field law given by~\eqref{eq:cucker-smale_mean_field}.
For fixed $(\mfldis_t)_{t\geq 0}$,
the first two equations of the synchronously coupled system~\eqref{eq:cucker-smale_mean_field_synch} may be viewed as a non-autonomous differential equation in~$\R^{2dJ}$,
whereas~\eqref{eq:cucker-smale} is an autonomous differential equation in the same space.
To these equations, we associate evolution operators~$\evl_{t_0, t}$ and~$\ev_{t-t_0}$,
both defined on $\R^{2dJ}$,
which map initial conditions at time~$t_0$ to solutions at time~$t$.
Let~$\allz{J}_0 \in \R^{2dJ}$ represent some initial data containing both positions and velocities.
The strategy is then to introduce a sequence of times $0 = t_0 \leq \dots \leq t_N$,
as well as the telescoping sum
\begin{equation}
    \label{eq:decomp_global_error}
    \ev_{t_N} \bigl(\allz{J}_0\bigr)  - \evl_{0, t_N} \bigl(\allz{J}_0\bigr)
    = \sum_{n = 0}^{N-1}  \ev_{t_N - t_{n}} \circ \evl_{0, t_n} \bigl(\allz{J}_0\bigr) - \ev_{t_N - t_{n+1}} \circ \evl_{0, t_{n+1}} \bigl(\allz{J}_0\bigr),
\end{equation}
where $t_n = n$ for all $n \in \N$.
See~\cref{fig:illustration_proof} for an illustration of the decomposition~\eqref{eq:decomp_global_error}.
We then rewrite each term of the sum as
\[
    \ev_{t_N - t_{n+1}} \circ \ev_{t_{n+1} - t_{n}} \circ \evl_{0, t_{n}} \bigl( \allz{J}_0 \bigr)
    - \ev_{t_N - t_{n+1}} \circ \evl_{t_n, t_{n+1}} \circ \evl_{0, t_{n}} \bigl(\allz{J}_0\bigr).
\]
This rewriting reveals that,
in order to bound the left-hand side of~\eqref{eq:decomp_global_error},
it is crucial to understand the difference between the operators $\ev_{t_{n+1} - t_{n}}$ and $\evl_{t_{n}, t_{n+1}}$,
which we achieve by using the finite-time propagation of chaos result~\cref{theorem:local_mean_field},
here playing the role of a consistency estimate.
The stability estimate~\cref{theorem:stability_ips} and the decay estimate~\cref{theorem:exponential_concenttration} then enable to control
the effect of the propagator~$\ev_{t_N - t_{n+1}}$ and the size of the initial data~$\evl_{0, t_{n}} \bigl(\allz{J}_0\bigr)$,
respectively.

\begin{figure}[ht]
    \centering
    \begin{tikzpicture}[scale=1.8, thick]
        \node[text=black] at (3, 3.5)
        {\small
        $\textcolor{TabOrange}{\Znn{J}{0}_T} - \textcolor{TabBlue}{\overline {\mathcal Z}^J_T}
            = {\color{TabGreen} \underbrace{\color{black} \textcolor{TabOrange}{\Znn{J}{0}_T} - \textcolor{TabOrange}{\Znn{J}{1}_T}}_{E_1}}
            + {\color{TabPurple} \underbrace{\color{black} \textcolor{TabOrange}{\Znn{J}{1}_T} - \textcolor{TabOrange}{\Znn{J}{2}_T}}_{E_2}}
            + \dots$};

        \draw[->] (-0.5, 0) -- (6.0, 0) node[below] {Time $t$};
        \draw[->] (0, -0.5) -- (0, 4) node[above] {};
        \draw (5, 0.0) -- (5, -0.05);
        \node[below] at (5, 0) {$T$};
        \node[rotate=90] at (-0.6, 2) {Configuration $\mathcal Z$};

        \node[text=TabBlue] at (1.3, .2) {Mean-field system};
        \node[text=TabBlue] at (.3, 2) {$\overline{\mathcal Z}^J_t$};
        \draw[domain=0:5,smooth,variable=\x,TabBlue,thick]
        plot ({\x},{3*exp(-0.7*\x)});

        \node[text=TabOrange] at (3, 2.5) {Interacting particle system};
        \draw[domain=0:5,smooth,variable=\x,TabOrange,thick]
        plot ({\x},{(3*exp(-0.2*\x))});
        \node[text=TabOrange] at (1, 2.70) {\small $\Znn{J}{0}_{t}$};

        \draw[domain=1:5,smooth,variable=\x,TabOrange,thick]
        plot ({\x},{(3*exp(-0.7)*exp(-0.2*(\x-1)))}) node[above] {};
        \node[text=TabOrange] at (1.5, 1.55) {\small $\Znn{J}{1}_{t}$};

        \draw[domain=2:5,smooth,variable=\x,TabOrange,thick]
        plot ({\x},{(3*exp(-0.7*2)*exp(-0.2*(\x-2)))}) node[right] {};
        \node[text=TabOrange] at (2.5, .85) {\small $\Znn{J}{2}_{t}$};

        \draw[domain=3:5,smooth,variable=\x,TabOrange,thick]
        plot ({\x},{(3*exp(-0.7*3)*exp(-0.2*(\x-3)))}) node[right] {};

        \draw[domain=4:5,smooth,variable=\x,TabOrange,thick]
        plot ({\x},{(3*exp(-0.7*4)*exp(-0.2*(\x-4)))}) node[right] {};

        \foreach \x in {0, 1, 2, 3, 4} {
                \node[circle, fill=TabBlue, inner sep=1pt] at (\x, {3*exp(-0.7*\x)}) {};
            }

        \draw[decorate,decoration={brace,amplitude=10pt},thick,color=TabGreen] (5.02, 1.11) -- (5.02, .66);
        \node[text=TabBlue,color=TabGreen] at (6.2, .9) {$\textcolor{TabOrange}{\Znn{J}{0}_T} - \textcolor{TabOrange}{\Znn{J}{1}_T} \eqcolon \textcolor{TabGreen}{E_1}$};
        \draw[decorate,decoration={brace,amplitude=10pt},thick,color=TabPurple] (5.02, .66) -- (5.02, .41);
        \node[text=TabBlue,color=TabPurple] at (6.2, .55) {$\textcolor{TabOrange}{\Znn{J}{1}_T} - \textcolor{TabOrange}{\Znn{J}{2}_T} \eqcolon \textcolor{TabPurple}{E_2}$};
    \end{tikzpicture}
    \caption{
        Illustration of the decomposition~\eqref{eq:decomp_global_error}, which is the main idea of the proof of~\cref{theorem:uit_prop_of_chaos}.
        Here, we used the notation
        $\Znn{J}{n}_t = \ev_{t - t_{n}} \circ \evl_{0, t_{n}} \bigl(\allz{J}_0\bigr)$,
        for $t \geq t_n$, as well as
        $\overline {\mathcal Z}^J_t = \evl_{0, t} \bigl(\allz{J}_0\bigr)$.
        The key idea of the proof is to break down the total error as a sum of small contributions $E_1, E_2, \dotsc$,
        each arising from the mismatch between the interacting particle system and the mean field particles over a bounded time interval
        $[t_i, t_{i+1}]$.
    }
    \label{fig:illustration_proof}
\end{figure}

Having outlined the key idea,
we now turn to the proof.
Observe that the following argument uses only the local-in-time propagation of chaos result~\cref{theorem:local_mean_field} (together with its consequence~\cref{corollary:mf}),
the uniform stability estimate~\cref{theorem:stability_ips}, and the exponential concentration estimate~\cref{theorem:exponential_concenttration},
so that our strategy applies in the same way to other systems of interacting particles for which similar results can be obtained.

\begin{proof}[Proof of \cref{theorem:uit_prop_of_chaos}]
    To prove the result,
    we repeatedly apply the finite-time propagation-of-chaos estimate given in~\cref{theorem:local_mean_field}.
    To this end,
    let us introduce for all $n \in \N$ the time $t_n = n$ and
    the solution $(\xnn{n}{j}, \vnn{n}{j})_{j \in \range{1}{J}}$ to
    the Cucker--Smale equation~\eqref{eq:cucker-smale} over the time interval $[t_n, \infty)$
    and with initial condition
    \[
        \left(\xnn{n}{j}_{t_n}, \vnn{n}{j}_{t_n} \right) =
        \left(\xnl{j}_{t_n}, \vnl{j}_{t_n}\right).
    \]
    Note in particular that $\bigl(\xn{j}, \vn{j}\bigr) = \bigl(\xnn{0}{j}, \vnn{0}{j}\bigr)$.
    Fix $t \in [0, \infty)$ and let $N = \lfloor t \rfloor$.
    By the triangle inequality,
    we have
    \begin{align}
        \notag
        \left(\expect  \Bigl\lvert \Delta \xnn{0}{j}_t - \Delta \xnl{j}_t \Bigr\rvert^2 + \expect \Bigl\lvert \Delta \vnn{0}{j}_t - \Delta \vnl{j}_t \Bigr\rvert^2 \right)^{\frac{1}{2}}
         & \le \sum_{n=0}^{N-1}
        \left(\expect  \Bigl\lvert \Delta \xnn{n}{j}_t - \Delta \xnn{n+1}{j}_t \Bigr\rvert^2
        + \expect  \Bigl\lvert \Delta \vnn{n}{j}_t - \Delta \vnn{n+1}{j}_t \Bigr\rvert^2 \right)^{\frac{1}{2}} \\
        \label{eq:main}
         & \qquad + \left(\expect \Bigl\lvert \Delta \xnn{N}{j}_t - \Delta \xnl{j}_t \Bigr\rvert^2
        +  \expect \Bigl\lvert \Delta \vnn{N}{j}_t - \Delta \vnl{j}_t \Bigr\rvert^2 \right)^{\frac{1}{2}}.
    \end{align}
    Using~\cref{corollary:mf},
    which is a corollary of~\cref{theorem:local_mean_field},
    we bound the square of the second term as follows:
    \[
        \expect  \bigl\lvert \Delta \xnn{N}{j}_t - \Delta \xnl{j}_t \bigr\rvert^2
        + \expect \bigl\lvert \Delta \vnn{N}{j}_t - \Delta \vnl{j}_t \bigr\rvert^2
        \leq  \frac{ 2 C_{\rm MF}}{J} \, \expect \left\lvert \vnl{1}_{t_N} - \expect \left[ \vnl{1}_{t_N} \right]  \right\rvert^2.
    \]
    For the terms in the sum,
    fix $n \in \range{0}{N-1}$ and note that the process $(\Delta \xnn{n}{j}_t, \Delta \vnn{n}{j}_t)_{j \in \range{1}{J}}$,
    for~$t \in [t_{n+1}, \infty)$, is solution to the Cucker--Smale equation~\eqref{eq:cucker-smale} with initial condition at~$t_{n+1}$
    given by $(\Delta \xnn{n}{j}_{t_{n+1}}, \Delta \vnn{n}{j}_{t_{n+1}})_{j \in \range{1}{J}}$.
    Furthermore, it holds that
    \[
        \frac{1}{J}\sum_{j=1}^{J} \Delta \vnn{n}{j}_{t_{n+1}} = 0,
        \qquad
        \frac{1}{J}\sum_{j=1}^{J} \Delta \vnn{n+1}{j}_{t_{n+1}} = 0,
    \]
    Therefore, we can use the uniform-in-time stability estimate in \cref{theorem:stability_ips},
    to deduce that
    \begin{align}
        \forall t \geq t_{n+1}, \qquad
         & \frac{1}{J}\sum_{j=1}^{J} \Bigl\lvert \Delta \xnn{n}{j}_t - \Delta \xnn{n+1}{j}_t \Bigr\rvert^2
        +  \frac{1}{J} \sum_{j=1}^{J} \Bigl\lvert \Delta \vnn{n}{j}_t - \Delta \vnn{n+1}{j}_t \Bigr\rvert^2 \\
         & \qquad \leq
        C_{\rm Stab}^2
        \left(
        \frac{1}{J}\sum_{j=1}^{J} \Bigl\lvert \Delta \xnn{n}{j}_{t_{n+1}} - \Delta \xnn{n+1}{j}_{t_{n+1}} \Bigr\rvert^2
        +   \frac{1}{J} \sum_{j=1}^{J} \Bigl\lvert \Delta \vnn{n}{j}_{t_{n+1}} - \Delta \vnn{n+1}{j}_{t_{n+1}} \Bigr\rvert^2 \right),
    \end{align}
    where $C_{\rm Stab} := C_{\rm Stab}\bigl(\diametermfl_X(0), \diametermfl_V(0), \diametermfl_V(0) \bigr)$.
    Taking the expectation and using exchangeability,
    we deduce that
    \begin{align}
         & \expect \Bigl\lvert \Delta \xnn{n}{j}_t - \Delta \xnn{n+1}{j}_t \Bigr\rvert^2
        +  \expect \Bigl\lvert \Delta \vnn{n}{j}_t - \Delta \vnn{n+1}{j}_t \Bigr\rvert^2 \\
         & \qquad \leq
        C_{\rm Stab}^2
        \left(
        \expect \Bigl\lvert \Delta \xnn{n}{j}_{t_{n+1}} - \Delta \xnn{n+1}{j}_{t_{n+1}} \Bigr\rvert^2
        +   \expect \Bigl\lvert \Delta \vnn{n}{j}_{t_{n+1}} - \Delta \vnn{n+1}{j}_{t_{n+1}} \Bigr\rvert^2
        \right)                                                                          \\
         & \qquad =
        C_{\rm Stab}^2
        \left(
        \expect \Bigl\lvert \Delta \xnn{n}{j}_{t_{n+1}} - \Delta \xnl{j}_{t_{n+1}} \Bigr\rvert^2
        +   \expect \Bigl\lvert \Delta \vnn{n}{j}_{t_{n+1}} - \Delta \vnl{j}_{t_{n+1}} \Bigr\rvert^2
        \right).
    \end{align}
    In the last equality,
    we used the fact that the processes $(\Delta \xnn{n}{j}_t, \Delta \vnn{n}{j}_t)$ are initialized at the mean-field particles at time $t = t_{n+1}$.
    To bound the right-hand side of this equation,
    we apply~\cref{theorem:local_mean_field} over the time interval $[t_n, t_{n+1}]$,
    which gives that
    \begin{align}
        \expect \Bigl\lvert \Delta \xnn{n}{j}_{t_{n+1}} - \Delta \xnl{j}_{t_{n+1}} \Bigr\rvert^2
        + \expect \Bigl\lvert \Delta \vnn{n}{j}_{t_{n+1}} - \Delta \vnl{j}_{t_{n+1}} \Bigr\rvert^2
         & \leq
        \frac{2 C_{\rm MF}}{J} \, \expect \left\lvert \vnl{1}_{t_n} - \expect \left[ \vnl{1}_{t_n} \right]  \right\rvert^2.
    \end{align}
    Combining the bounds in~\eqref{eq:main},
    we finally obtain
    \begin{align}
         & \left(\expect  \Bigl\lvert \Delta \xnn{0}{j}_t - \Delta \xnl{j}_t \Bigr\rvert^2 + \expect \Bigl\lvert \Delta \vnn{0}{j}_t - \Delta \vnl{j}_t \Bigr\rvert^2 \right)^{\frac{1}{2}} \\
         & \qquad \qquad
        \leq \frac{1}{\sqrt{J}} \sum_{n=0}^{N-1} \left(2 C_{\rm Stab}^2 C_{\rm MF} \, \expect \left\lvert \vnl{1}_{t_n} - \expect \left[ \vnl{1}_{t_n} \right]  \right\rvert^2 \right)^{\frac{1}{2}}
        + \frac{1}{\sqrt{J}} \left(2 C_{\rm MF} \, \expect \left\lvert \vnl{1}_{t_N} - \expect \left[ \vnl{1}_{t_N} \right]  \right\rvert^2 \right)^{\frac{1}{2}}.
    \end{align}
    It remains to bound the sum on the right-hand side independently of $N$.
    This is a direct consequence of~\cref{theorem:exponential_concenttration}
    and the formula for geometric series:
    \begin{align}
        \sum_{n=0}^{N-1} \left( \expect \left\lvert \vnl{1}_{t_n} - \expect \left[ \vnl{1}_{t_n} \right]  \right\rvert^2 \right)^{\frac{1}{2}}
         & \leq
        \sum_{n=0}^{N-1} C_{\rm Decay} \e^{- \lambda t_n} \left( \expect \left\lvert \vnl{1}_{0} - \expect \left[ \vnl{1}_{0} \right]  \right\rvert^2 \right)^{\frac{1}{2}}                   \\
         & \leq \left(1 + \frac{1}{\e^{\lambda} - 1}\right) C_{\rm Decay}  \left( \expect \left\lvert \vnl{1}_{0} - \expect \left[ \vnl{1}_{0} \right]  \right\rvert^2 \right)^{\frac{1}{2}},
    \end{align}
    The main claim~\eqref{eq:chaos_main} follows,
    and last claim on
    the Wasserstein distance
    is an immediate consequence,
    since a coupling between $\mathcal R \rho^J_t$ and~$\mathcal R \mfldis_t^{\otimes J}$
    is provided by the ensembles~$(\Delta \xn{j}_t, \Delta \vn{j}_t)_{j \in \range{1}{J}}$
    and~$(\Delta \xnl{j}_t, \Delta \vnl{j}_t)_{j \in \range{1}{J}}$.
\end{proof}

\section{Proofs of the auxiliary results}
\label{sec:aux-proofs}

We recall in this section relevant results from~\cite{MR2536440,MR2860672,HaKimZhang2018uniformstab}.

Recall that we introduced the notation $\Psi(x) = \psi(|x|)$ for conciseness later on.
The rest of this section is organized as follows.
We begin by proving the finite-time mean-field result~\cref{theorem:local_mean_field},
which is similar to~\cite[Theorem~1.1]{MR2860672},
except that we assume the velocity marginal of~$\mfldis_0$ is compactly supported,
which enables to sharpen the convergence the convergence rate for $J^{-1 + \varepsilon}$ to~$J^{-1}$,
and that we make explicit all factors on the right-hand side.
After stating a simple corollary of this theorem used in the proof of our main result,
we then present in~\cref{theorem:stability_ips} a uniform-in-time stability estimate for the particle system~\eqref{eq:cucker-smale}.
We finally recall in~\cref{theorem:exponential_concenttration} a decay estimate for the nonlinear mean-field dynamics.

\subsection{Proof of the finite-time mean-field result \texorpdfstring{{\cref{theorem:local_mean_field}}}{}}
\label{sec:proof-local-mean-field}
\begin{proof}[Proof of \cref{theorem:local_mean_field}]
    We follow closely the proof of~\cite{MR2860672}.
    Let $\dxn{j}_t = \xn{j}_t - \xnl{j}_t$
    and $\dvn{j}_t = \vn{j}_t - \vnl{j}_t$.
    By Leibniz's rule, it holds that
    \begin{equation}
        \label{eq:first_equation}
        \frac{1}{2}
        \frac{\d}{\d t} \Bigl( \lvert \dxn{j}_t \rvert^2 + \lvert \dvn{j}_t \rvert^2 \Bigr)
        = \dxn{j}_t \cdot \dvn{j}_t
        - \frac{K}{J} \sum_{k=1}^{J} \dvn{j}_t \cdot \left(  H \Bigl(  \xn{j}_t - \xn{k}_t, \vn{j}_t - \vn{k}_t \Bigr)
        - H \star \mfldis_t \Bigl(  \xnl{j}_t, \vnl{j}_t \Bigr) \right).
    \end{equation}
    The first term is bounded from above by $\frac{1}{2} \bigl( \lvert \dxn{j}_t \rvert^2 + \lvert \dvn{j}_t \rvert^2 \bigr)$ by Young's inequality.
    We rewrite the second term as
    \begin{align}
         & -\frac{K}{J} \sum_{k=1}^{J}  \dvn{j}_t \cdot \biggl( H \Bigl(  \xn{j}_t - \xn{k}_t, \vn{j}_t - \vn{k}_t \Bigr)
        - H \star \mfldis_t \Bigl(  \xnl{j}_t, \vnl{j}_t \Bigr) \biggr)                                                   \\
         & \qquad \qquad =
        -\frac{K}{J} \sum_{k=1}^{J} \dvn{j}_t \cdot \biggl(  H \Bigl(  \xn{j}_t - \xn{k}_t, \vn{j}_t - \vn{k}_t \Bigr)
        - H \Bigl(  \xnl{j}_t - \xnl{k}_t, \vnl{j}_t - \vnl{k}_t \Bigr) \biggr)                                           \\
         & \qquad \qquad \qquad
        - \frac{K}{J} \sum_{k=1}^{J}  \dvn{j}_t \cdot \biggl( H \Bigl(  \xnl{j}_t - \xnl{k}_t, \vnl{j}_t - \vnl{k}_t \Bigr)
        - H \star \mfldis_t \Bigl(  \xnl{j}_t, \vnl{j}_t \Bigr) \biggr) =: K A_t + K B_t.
    \end{align}
    \paragraph{Step 1. Bound of $A_t$}
    Recall that $H(x, v) = \Psi(x) v$ with $\Psi(x) = \Psi(-x)$, and so $H(-x, -v) = - H(x, v)$.
    Therefore, using in addition exchangeability, we have that
    \begin{align}
        \expect \left[A_t\right] & =
        -\expect \left[ \frac{1}{J^2} \sum_{j=1}^{J} \sum_{k=1}^{J} \dvn{j}_t \cdot \biggl( H \Bigl(  \xn{j}_t - \xn{k}_t, \vn{j}_t - \vn{k}_t \Bigr)
        - H \Bigl(  \xnl{j}_t - \xnl{k}_t, \vnl{j}_t - \vnl{k}_t \Bigr) \biggr) \right]                                                                                                                        \\
                                 & = -\expect \left[ \frac{1}{2J^2} \sum_{j=1}^{J} \sum_{k=1}^{J} \left( \dvn{j}_t - \dvn{k}_t \right) \cdot \biggl( H \Bigl(  \xn{j}_t - \xn{k}_t, \vn{j}_t - \vn{k}_t \Bigr)
            - H \Bigl(  \xnl{j}_t - \xnl{k}_t, \vnl{j}_t - \vnl{k}_t \Bigr) \biggr) \right].
    \end{align}
    Adding and subtracting a middle term $H \bigl(  \xnl{j}_t - \xnl{k}_t, \vn{j}_t - \vn{k}_t \bigr)$ within the bracket,
    and using the definition of $H$,
    we deduce that
    \begin{align}
        \expect \left[A_t\right]
         & = - \expect \left[ \frac{1}{2J^2} \sum_{j=1}^{J} \sum_{k=1}^{J} \left( \dvn{j}_t - \dvn{k}_t \right) \cdot
        \left(\vn{j}_t - \vn{k}_t\right) \left( \Psi \Bigl( \xn{j}_t - \xn{k}_t \Bigr) - \Psi \Bigl( \xnl{j}_t - \xnl{k}_t \Bigr) \right)   \right]                                  \\
         & \qquad - \expect \left[ \frac{1}{2J^2} \sum_{j=1}^{J} \sum_{k=1}^{J} \Psi \Bigl( \xnl{j}_t - \xnl{k}_t \Bigr)  \left\lvert \dvn{j}_t - \dvn{k}_t \right\rvert^2  \right].
    \end{align}
    The second term is negative.
    To bound the first term,
    note that by Lipschitz continuity of~$\psi$ and by the reverse triangle inequality,
    it holds that
    \[
        \left\lvert \Psi \Bigl( \xn{j}_t - \xn{k}_t \Bigr) - \Psi \Bigl( \xnl{j}_t - \xnl{k}_t \Bigr) \right\rvert
        =
        \left\lvert \psi \Bigl( \left\lvert \xn{j}_t - \xn{k}_t \right\rvert \Bigr) - \psi \Bigl( \left\lvert \xnl{j}_t - \xnl{k}_t \right\rvert \Bigr) \right\rvert
        \leq \lip \left\lvert \dxn{j}_t - \dxn{k}_t \right\rvert.
    \]
    It follows that
    \begin{align}
        \notag
        \expect \left[A_t\right]
         & \leq  \expect \left[ \frac{\lip}{2J^2} \sum_{j=1}^{J} \sum_{k=1}^{J}  \left\lvert \dvn{j}_t - \dvn{k}_t \right\rvert
            \times \left\lvert \vn{j}_t - \vn{k}_t \right\rvert \times
        \left\lvert \dxn{j}_t - \dxn{k}_t \right\rvert \right]                                                                  \\
        \label{eq:bound_At}
         & \leq
        \expect \left[ \frac{\lip \diameter_V(t)}{2J^2} \sum_{j=1}^{J} \sum_{k=1}^{J}  \left\lvert \dvn{j}_t - \dvn{k}_t \right\rvert
            \times \left\lvert \dxn{j}_t - \dxn{k}_t \right\rvert \right]
        \leq \lip {\overline \diameter_V}(0) \left( \expect \Bigl\lvert \dxn{1}_t \Bigr\rvert^2 + \expect \Bigl\lvert \dvn{1}_t \Bigr\rvert^2 \right).
    \end{align}
    Here we introduced $\diameter_V(t) := \max_{j,k \in \range{1}{J}} \bigl\lvert \vn{j}_t -  \vn{k}_t \bigr\rvert$
    and used that $\diameter_V(t) \leq \diameter_V(0) \leq \overline \diameter_V(0)$ for all $t \geq 0$.
    The inequality~$\diameter_V(t) \leq \diameter_V(0)$ is a standard property of the Cucker--Smale dynamics;
    see for example~\cite[Theorem 3.4]{motsch2011new}.
    We also used Young's inequality together with the elementary bound $(a + b)^2 \leq 2a^2 + 2b^2$.
    \paragraph{Step 2. Bound of $B_t$}
    We follow the classical approach,
    but keep track of the constant prefactors.
    By Young's inequality, it holds that
    \[
        2\expect [B_t] \leq
        \expect \left\lvert \dvn{1}_t \right\rvert^2
        + \expect \biggl\lvert  \frac{1}{J} \sum_{k=1}^{J}
        H \Bigl(  \xnl{1}_t - \xnl{k}_t, \vnl{1}_t - \vnl{k}_t \Bigr) - H \star \mfldis_t \Bigl(  \xnl{1}_t, \vnl{1}_t \Bigr) \biggr \rvert^2.
    \]
    Apart from the first term in the sum which is 0,
    the summands are uncorrelated with expectation 0,
    and so after expanding the sum the cross terms cancel out,
    which gives
    \begin{align}
        2\expect [B_t]
         & \leq
        \expect \left\lvert \dvn{1}_t \right\rvert^2
        + \frac{1}{J} \expect \biggl\lvert
        H \Bigl(  \xnl{1}_t - \xnl{2}_t, \vnl{1}_t - \vnl{2}_t \Bigr) - H \star \mfldis_t \Bigl(  \xnl{1}_t, \vnl{1}_t \Bigr) \biggr \rvert^2.
    \end{align}
    Writing the expectation in terms of the mean-field law,
    and then using the property of the mean as the best $L^2$ approximation by a constant,
    we deduce that
    \begin{align}
        2\expect [B_t]
         & \leq
        \expect \left\lvert \dvn{1}_t \right\rvert^2
        + \frac{1}{J} \int_{\R^{2d}} \int_{\R^{2d}}
        \biggr \lvert H \Bigl(  x - y, v - w \Bigr) - \int_{\R^{2d}} H \Bigl(  x - Y, v - W \Bigr) \, \mfldis_t (\d Y \, \d W) \biggr \rvert^2
        \mfldis_t(\d y \, \d w) \, \mfldis_t(\d x \, \d v).
        \\
         & \leq
        \expect \left\lvert \dvn{1}_t \right\rvert^2
        + \frac{1}{J} \int_{\R^{2d}} \int_{\R^{2d}}
        \biggr \lvert H \Bigl(  x - y, v - w \Bigr) \biggr \rvert^2
        \mfldis_t(\d y \, \d w) \, \mfldis_t(\d x \, \d v).
    \end{align}
    Since $\norm{\psi}_{L^{\infty}} \leq 1$,
    we finally obtain
    \begin{align}
        \notag
        2\expect [B_t]
         & \leq
        \expect \left\lvert \dvn{1}_t \right\rvert^2
        + \frac{1}{J} \int_{\R^{2d}} \int_{\R^{2d}}
        \lvert v - w \rvert^2 \, \mfldis_t(\d y \, \d w) \, \mfldis_t(\d x \, \d v) \\
        \label{eq:bound_Bt}
         & \leq
        \expect \left\lvert \dvn{1}_t \right\rvert^2
        + \frac{2}{J} \expect \left\lvert \vnl{1}_t - \expect \left[ \vnl{1}_t \right]  \right\rvert^2,
    \end{align}
    where we used that
    \[
        \expect  \left\lvert \vnl{1}_t - \vnl{2}_t \right\rvert^2
        = 2\expect \left\lvert \vnl{1}_t - \expect \left[ \vnl{1}_t \right]  \right\rvert^2.
    \]

    \paragraph{Step 3. Gr\"onwall's inequality}
    Combining~\eqref{eq:first_equation} with~\eqref{eq:bound_At} and~\eqref{eq:bound_Bt},
    we deduce that
    \begin{equation}
        \label{eq:before_gronwall}
        \frac{1}{2}
        \frac{\d}{\d t} \expect \Bigl[ \lvert \dxn{j}_t \rvert^2 + \lvert \dvn{j}_t \rvert^2 \Bigr]
        \leq
        \left({\frac{1}{2}} + K \lip {\overline \diameter_V}(0) + \frac{K}{2}  \right) \left( \expect \bigl\lvert \dxn{1}_t \bigr\rvert^2 + \expect \bigl\lvert \dvn{1}_t \bigr\rvert^2 \right)
        + {\frac{K}{J}} \expect \Bigl\lvert \vnl{1}_t - \expect \left[ \vnl{1}_t \right]  \Bigr\rvert^2.
    \end{equation}
    From~\eqref{eq:cucker-smale_mean_field},
    we have that
    \begin{align}
        \notag
        \frac{\d}{\d t} \expect \left[ \left\lvert \vnl{1}_t - \expect \left[ \vnl{1}_t \right] \right\rvert^2 \right]
         & = - K \expect \left[ \left(\vnl{1}_t - \expect \left[ \vnl{1}_t \right]\right) \cdot \bigl(H \star \mfldis_t\bigr) \Bigl(\xnl{1}_t, \vnl{1}_t\Bigr) \right] \\
        \notag
         & = - K \expect \left[ \left(\vnl{1}_t - \expect \left[ \vnl{1}_t \right]\right) \cdot H\left(\xnl{1}_t - \xnl{2}_t, \vnl{1}_t - \vnl{2}_t\right) \right]     \\
        \label{eq:symmetry_used_here}
         & = - \frac{K}{2} \expect \left[ \left(\vnl{1}_t - \vnl{2}_t \right) \cdot H\left(\xnl{1}_t - \xnl{2}_t, \vnl{1}_t - \vnl{2}_t\right) \right]                 \\
        \notag
         & = - \frac{K}{2} \expect \left[  \Psi\left(\xnl{1}_t - \xnl{2}_t \right)
            \left\lvert  \vnl{1}_t - \vnl{2}_t \right \rvert^2 \right] \leq 0,
    \end{align}
    where we used again the symmetry $H(-x, -v) = - H(x, v)$ {to deduce~\eqref{eq:symmetry_used_here}.}
    Thus, $\expect \bigl\lvert \vnl{1}_t - \expect \bigl[ \vnl{1}_t \bigr]  \bigr\rvert^2 \leq \expect \bigl\lvert \vnl{1}_0 - \expect \left[ \vnl{1}_0 \right]  \bigr\rvert^2$,
    and so by~\eqref{eq:before_gronwall} and Gr\"onwall's differential inequality,
    we obtain
    \[
        \expect \Bigl[ \lvert \dxn{j}_t \rvert^2 + \lvert \dvn{j}_t \rvert^2 \Bigr]
        \leq
        \frac{{2K}}{J} \left( \exp \Bigl(\bigl({1} + 2 K \lip {\overline \diameter_V}(0) + K\bigr) t \Bigr) - 1\right) \frac{\expect \left\lvert \vnl{1}_0 - \expect \left[ \vnl{1}_0 \right]  \right\rvert^2}{{1} + 2 K \lip {\overline \diameter_V}(0) + K},
    \]
    which leads to the statement.
\end{proof}

Observe that since we work with the assumption of compactly supported initial velocities,
the above proof does not require to partition the state space in $\{\vnl{1}_t>R\}$ and $\{\vnl{1}_t\le R\}$ as in the proof of Theorem 1.1 in~\cite{MR2860672}.
Also, our convergence rate in $J$ is the Monte Carlo rate $\frac{1}{J}$, which is slightly better than the rate $\frac{1}{J^{1-\eps}}$ in~\cite{MR2860672}.

\begin{corollary}
    \label{corollary:mf}
    Under the same assumptions as in~\cref{theorem:local_mean_field},
    it holds that
    \[
        \expect \left[ \bigl\lvert \Delta \xn{j}_t - \Delta \xnl{j}_t \bigr\rvert^2 + \bigl\lvert \Delta \vn{j}_t - \Delta \vnl{j}_t \bigr\rvert^2 \right]
        \leq
        \frac{2 \bra*{C_{\rm MF}}^t}{J} \, \expect \left\lvert \vnl{1}_0 - \expect \left[ \vnl{1}_0 \right]  \right\rvert^2,
    \]
    with the same constant~$C_{\rm MF}$ as in~\cref{theorem:local_mean_field},
    and with the notation
    \[
        \Delta \xn{j}_t = \xn{j}_t - \frac{1}{J} \sum_{j=1}^{J} \xn{j}_t,
        \qquad
        \Delta \xnl{j}_t = \xnl{j}_t - \frac{1}{J} \sum_{j=1}^{J} \xnl{j}_t,
        \qquad
        \Delta \vn{j}_t = \vn{j}_t - \frac{1}{J} \sum_{j=1}^{J} \vn{j}_t,
        \qquad
        \Delta \vnl{j}_t = \vnl{j}_t - \frac{1}{J} \sum_{j=1}^{J} \vnl{j}_t.
    \]
\end{corollary}
\begin{proof}
    It follows from \eqref{eq:cucker-smale} that the mean velocity $\frac{1}{J} \sum_{j=1}^{J} \vn{j}_t$ is constant in time.
    By exchangeability,
    \begin{align}
        \expect \left[ \bigl\lvert \Delta \xn{1}_t - \Delta \xnl{1}_t \bigr\rvert^2 + \bigl\lvert \Delta \vn{1}_t - \Delta \vnl{1}_t \bigr\rvert^2 \right]
         & = \expect \left[ \frac{1}{J}\sum_{j=1}^{J} \bigl\lvert \Delta \xn{j}_t - \Delta \xnl{j}_t \bigr\rvert^2
        + \frac{1}{J} \sum_{j=1}^{J} \bigl\lvert \Delta \vn{j}_t - \Delta \vnl{j}_t \bigr\rvert^2 \right]          \\
         & \leq
        \expect \left[ \frac{1}{J} \sum_{j=1}^{J} \bigl\lvert \xn{j}_t - \xnl{j}_t \bigr\rvert^2
            + \frac{1}{J}\sum_{j=1}^{J} \bigl\lvert \vn{j}_t - \vnl{j}_t \bigr\rvert^2 \right].
    \end{align}
    Here we used that,
    for any collection $(z_1, \dotsc, z_J)$ in~$\R^d$,
    the sum $\frac{1}{J} \sum_{j=1}^{J} \lvert z_j - a \rvert^2$ is minimized when $a = \frac{1}{J} \sum_{j=1}^{J} z_j$.
    The statement is then a simple consequence of~\cref{theorem:local_mean_field}.
\end{proof}

\subsection{Proof of the uniform-in-time stability estimate \texorpdfstring{{\cref{theorem:stability_ips}}}{}}
\label{sec:proof-stability}

\begin{proof}[Proof of \cref{theorem:stability_ips}]
    We largely follow the proof given in~\cite[Theorem 1.2]{HaKimZhang2018uniformstab}.
    The first step is to obtain differential inequalities of the form \cite[Lemma 2.2]{HaKimZhang2018uniformstab}
    \begin{align}
        \left\lvert \frac{\d}{\d t} \diameter_X(t) \right\rvert & \leq \diameter_V(t),
        \qquad
        {\frac{\d}{\d t} \diameter_V(t) \leq - K \psi\bigl(\diameter_X(t)\bigr) \diameter_V(t)}
    \end{align}
    for the diameters \eqref{eq:diameters1}
    and to deduce from these that,
    if~$K$ satisfies~\eqref{eq:assump_K_stab},
    then
    \begin{equation}
        \label{eq:bounds_diameters}
        \sup_{t \geq 0} \diameter_X(t) \leq x_{\infty}, \qquad
        \diameter_V(t) \leq \diameter_V(0) \e^{- K \psi(x_{\infty}) t},
    \end{equation}
    where $x_{\infty}$ is defined from the equation
    \begin{equation}
        \label{eq:x_inf}
        \int_{\diameter_X(0)}^{x_{\infty}} \psi(s) \, \d s = \frac{\diameter_V(0)}{K},
    \end{equation}
    which is indeed well-defined given the assumption~\eqref{eq:assump_K_stab} on~$K$.
    With these estimates,
    the rest of the proof is as follows.
    Let us introduce
    \[
        \mathcal L_X(t) := \frac{1}{2J} \sum_{j=1}^{J} \Bigl\lvert \xn{j}_t - \xnt{j}_t \Bigr\rvert^2,
        \qquad
        \mathcal L_V(t) := \frac{1}{2J} \sum_{j=1}^{J} \Bigl\lvert \vn{j}_t - \vnt{j}_t \Bigr\rvert^2.
    \]
    It is simple to show that $\frac{\d}{\d t} \mathcal L_X \leq 2 \sqrt{\mathcal L_X} \sqrt{\mathcal L_V}$.
    On the other hand,
    \begin{align}
        \frac{\d}{\d t} \mathcal L_V(t)
         & = \frac{K}{J^2} \sum_{j=1}^{J} \sum_{k=1}^{J} \Bigl\langle \vn{j}_t - \vnt{j}_t, \Psi\bigl(\xn{j}_t- \xn{k}_t\bigr) (\vn{k}_t - \vn{j}_t)
        - \Psi\bigl(\xnt{j}_t- \xnt{k}_t\bigr) (\vnt{k}_t - \vnt{j}_t) \Bigr\rangle
        \\
         & =
        \frac{K}{J^2} \sum_{j=1}^{J} \sum_{k=1}^{J}
        \Psi\bigl(\xn{j}_t- \xn{k}_t\bigr)
        \Bigl\langle \vn{j}_t - \vnt{j}_t, \vn{k}_t - \vn{j}_t - \vnt{k}_t + \vnt{j}_t \Bigr\rangle
        \\
         & \qquad \qquad
        + \frac{K}{J^2} \sum_{j=1}^{J} \sum_{k=1}^{J} \left(\Psi\bigl(\xn{j}_t- \xn{k}_t\bigr) - \Psi\bigl(\xnt{j}_t-\xnt{k}_t\bigr)\right) \Bigl\langle \vn{j}_t - \vnt{j}_t, \vnt{k}_t - \vnt{j}_t \Bigr\rangle
        =: \mathcal A_1 + \mathcal A_2.
    \end{align}
    For the first term on the right-hand side,
    using $H(-x, -v) = -H(x, v)$ we have
    \begin{align}
        \mathcal A_1
         & = - \frac{K}{2J^2} \sum_{j=1}^{J} \sum_{k=1}^{J} \Psi\bigl(\xn{j}_t- \xn{k}_t\bigr) \left\lvert \vn{k}_t - \vnt{k}_t - \vn{j}_t + \vnt{j}_t \right\rvert^2                                              \\
         & \leq - \frac{K}{2J^2} \left(\min_{j,k \in \range{1}{J}} \Psi\bigl(\xn{j}_t- \xn{k}_t\bigr) \right) \sum_{j=1}^{J} \sum_{k=1}^{J} \left\lvert \vn{k}_t - \vnt{k}_t - \vn{j}_t + \vnt{j}_t \right\rvert^2 \\
         & \leq - \frac{K \psi\bigl(\diameter_X(t)\bigr) }{J} \sum_{k=1}^{J} \left\lvert \vn{k}_t - \vnt{k}_t \right\rvert^2,
    \end{align}
    where in the last inequality,
    we used that
    \[
        \frac{1}{J^2} \sum_{j=1}^{J} \sum_{k=1}^{J} \left\lvert \vn{k}_t - \vnt{k}_t - \vn{j}_t + \vnt{j}_t \right\rvert^2
        = \frac{2}{J}\sum_{j=1}^{J} \left\lvert \vn{j}_t - \vnt{j}_t \right\rvert^2.
    \]
    For the second term,
    the Cauchy-Schwarz and Young inequalities,
    together with Lipschitz continuity of~$\psi$ and the reverse triangle inequality,
    give
    \begin{align}
        \mathcal A_2
         & \leq
        \frac{K}{J^2} \sum_{j=1}^{J} \sum_{k=1}^{J} \left\lvert\Psi\bigl(\xn{j}_t- \xn{k}_t\bigr) - \Psi\bigl(\xnt{j}_t- \xnt{k}_t\bigr)\right\rvert
        \cdot \Bigl\lvert \vn{j}_t - \vnt{j}_t \Bigr\rvert \cdot \left\lvert \vnt{k}_t - \vnt{j}_t \right\rvert                                                                                                         \\
         & \leq
        \sqrt{\frac{K}{J} \sum_{j=1}^{J}  \Bigl\lvert \vn{j}_t - \vnt{j}_t \Bigr\rvert^2}
        \sqrt{\frac{K}{J^2} \sum_{j=1}^{J} \sum_{k=1}^{J}  \left\lvert \Psi\bigl(\xn{j}_t- \xn{k}_t\bigr) - \Psi\bigl(\xnt{j}_t- \xnt{k}_t\bigr) \right\rvert^2 \cdot \left\lvert \vnt{k}_t - \vnt{j}_t \right\rvert^2} \\
         & \leq \sqrt{2 K \mathcal L_V(t)}
        \sqrt{\frac{K L_{\psi}^2}{J^2} \sum_{j=1}^{J} \sum_{k=1}^{J}  \left\lvert \xn{k}_t - \xn{j}_t - \xnt{k}_t + \xnt{j}_t \right\rvert^2 \diameter_{\widetilde V}(t)^2}                                             \\
         & \leq
        {\sqrt{8}}
        K L_{\psi}\diameter_{\widetilde V}(t)
        \sqrt{\mathcal L_V(t)}
        \sqrt{\mathcal L_X(t)},
    \end{align}
    where again we used that
    \[
        \frac{1}{J^2} \sum_{j=1}^{J} \sum_{k=1}^{J}  \left\lvert \xn{k}_t - \xn{j}_t - \xnt{k}_t + \xnt{j}_t \right\rvert^2
        = \frac{2}{J} \sum_{j=1}^{J}  \left\lvert \xn{j}_t - \xnt{j}_t - \frac{1}{J} \sum_{j=1}^{J} \Bigl(\xn{j}_t - \xnt{j}_t\Bigr) \right\rvert^2
        \leq \frac{2}{J} \sum_{j=1}^{J}  \left\lvert \xn{j}_t - \xnt{j}_t \right\rvert^2.
    \]
    It follows from these inequalities that
    \begin{equation}
        \label{eq:differential_inequalities}
        \frac{\d}{\d t} \sqrt{\mathcal L_X} \leq \sqrt{\mathcal L_V},
        \qquad
        \frac{\d}{\d t} \sqrt{\mathcal L_V}
        \leq - K \psi\bigl(\mathcal D_X(t)\bigr) \sqrt{\mathcal L_V} + \sqrt{2} K L_{\psi}\diameter_{\widetilde V}(t) \sqrt{\mathcal L_X}.
    \end{equation}
    The bound~\eqref{eq:bounds_diameters} and a bootstrapping argument based on Gr\"onwall's lemma,
    given in~\cite[Lemma 3.1]{HaKimZhang2018uniformstab},
    could then be employed to prove~\eqref{eq:uniform_stab_statement}.
    However, a sharper estimate can be obtained more directly.
    To explain how, let us introduce~\(
    \alpha := K \psi(x_{\infty}),
    \)
    and
    \(
    \gamma := \sqrt{2} K L_{\psi} \diameter_{\widetilde V}(0).
    \)
    By~\eqref{eq:differential_inequalities},
    it holds that
    \[
        \frac{\d}{\d t} \left( \alpha \sqrt{\mathcal L_X} + \sqrt{\mathcal L_V} \right)
        \leq \frac{\gamma}{\alpha} \left( \alpha \sqrt{\mathcal L_X} + \sqrt{\mathcal L_V} \right) \e^{- \alpha t},
    \]
    which implies, via a Gr\"onwall inequality, that
    \[
        \left( \alpha \sqrt{\mathcal L_X(t)} + \sqrt{\mathcal L_V(t)} \right)
        \leq \exp \left( \frac{\gamma}{\alpha^2}\right) \left( \alpha \sqrt{\mathcal L_X(0)} + \sqrt{\mathcal L_V(0)} \right).
    \]
    Thus, since $\sqrt{a+b} \leq \sqrt{a} + \sqrt{b} \leq \sqrt{2} \sqrt{a + b}$ for any $a, b \in \R^+$,
    it follows that
    \begin{align}
        \sqrt{\mathcal L_X(t) + \mathcal L_V(t)}
         & \leq
        \sqrt{\mathcal L_X(t)} + \sqrt{\mathcal L_V(t)}                                                                                                        \\
         & \leq
        \max \left\{1, \frac{1}{\alpha} \right\}
        \left(\alpha \sqrt{\mathcal L_X(t)} + \sqrt{\mathcal L_V(t)} \right)                                                                                   \\
         & \leq \max \left\{1, \frac{1}{\alpha} \right\} \exp \left( \frac{\gamma}{\alpha^2}\right)
        \left(\alpha \sqrt{\mathcal L_X(0)} + \sqrt{\mathcal L_V(0)} \right)
        \\
         & \leq \sqrt{2} \max \left\{ \alpha, \frac{1}{\alpha} \right\} \exp \left( \frac{\gamma}{\alpha^2}\right) \sqrt{\mathcal L_X(0) + \mathcal L_V(0)}\,.
    \end{align}
    In order to give an expression of the function $C_{\rm Stab}$ as explicit as possible,
    define $\mathcal I (t):= \int_0^t \psi(s) \, \d s$.
    From \eqref{eq:x_inf} we have that $\mathcal I (x_\infty) = \mathcal I \bigl( \diameter_X(0) \bigr) + \frac{\diameter_V(0)}{K}$.
    Therefore, inequality~\eqref{eq:uniform_stab_statement} holds with constant
    \begin{align}
        C_{\rm Stab}\left(\diameter_X(0),  \diameter_V(0), \diametert_V(0)\right)
         & := \sqrt{2} \max \Bigl\{ K, \frac{1}{\alpha} \Bigr\} \exp \biggl( \frac{\sqrt{2} K L_{\psi} \diameter_{\widetilde V}(0)}{\alpha^2}\biggr),
    \end{align}
    where $\alpha = K\psi\Bigl(\mathcal I^{ -1}\bigl(\mathcal I( \diameter_X(0) ) + \frac{\diameter_V(0)}{K}\bigr)\Bigr)$ is a non-increasing function of $\diameter_X(0)$ and $\diameter_V(0)$.
\end{proof}

\section{Numerical illustration}
\label{sec:numerics}

In this section, we illustrate numerically the main result of~\cref{sec:Uniform-in-time_propagation_of_chaos}.
More precisely, the numerical experiment that follows aims at illustrating that there exists $C_{\rm Chaos}$ independent of~$J$ such that,
if $(\xn{j,J}_0, \vn{j,J}_0) = (\xnl{j}_0, \vnl{j}_0) \stackrel{\rm i.i.d.}{\sim} \mfldis_0$ for all $j \in \range{1}{J}$,
then
\begin{equation}
    \label{eq:num_illustrated}
    \forall t \geq 0, \qquad
    \sum_{j=1}^{J} \expect \left[ \Bigl\lvert \Delta \xn{j,J}_t - \Delta \xnl{j}_t \Bigr\rvert^2 + \Bigl\lvert \Delta \vn{j,J}_t - \Delta \vnl{j}_t \Bigr\rvert^2 \right]
    \leq C_{\rm Chaos}^2.
\end{equation}
Here we write $\xn{j,J}_t$ instead of $\xn{j}$ to emphasize that a system of size~$J$ is considered.
To illustrate~\eqref{eq:num_illustrated} numerically,
we need to simulate the mean-field dynamics and the interacting particle system.
Since it is not possible to simulate the mean-field dynamics exactly,
we calculate a precise approximation by simulating the finite-size Cucker--Smale system~\eqref{eq:cucker-smale} with a very large number~$J_{\infty}$ of particles.
Then, for various smaller values of~$J$,
we copy the ensemble composed of the first~$J$ mean-field particles at the initial time,
and evolve this ensemble as an interacting particle system of size~$J$ according to the dynamics~\eqref{eq:cucker-smale}.
To discretize the dynamics in time, the explicit Euler scheme with time step~$\Delta t$ is used in all cases.
During the simulation, we track the evolution of the quantities
\begin{subequations}
    \begin{align}
        \label{eq:errx}
        \texttt{errX(t)} & :=
        \sum_{j=1}^{J}  \biggl\lvert \xn{j,J}_t - \frac{1}{J} \sum_{j=1}^{J} \xn{j,J}_t - \xn{j,J_{\infty}}_t  + \sum_{j=1}^{J} \xn{j,J_{\infty}}_t \biggr\rvert^2, \\
        \label{eq:errv}
        \texttt{errV(t)} & :=
        \sum_{j=1}^{J} \biggl\lvert \vn{j,J}_t - \frac{1}{J} \sum_{j=1}^{J} \vn{j,J}_t - \vn{j,J_{\infty}}_t + \frac{1}{J} \sum_{j=1}^{J} \vn{j,J_{\infty}}_t \biggr\rvert^2,
    \end{align}
\end{subequations}
and approximate their expectation by a Monte Carlo approach based on~$M$ independent repetitions of the experiment.
All the parameters of the numerical experiment are collated in~\cref{table:training_params}.
In particular, note that we consider a regime where exponential flocking occurs and
condition~\eqref{eq:assump:K-large-enough} is satisfied.

\begin{table}[htb]
    \footnotesize
    \begin{center}
        \begin{tabular}{ccc}
            \toprule
            \textbf{Parameter}                                       & \textbf{Notation} & \textbf{Value}                                \\
            \midrule
            \phantom{\LARGE$\int$} Dimension                         & $d$               & 2                                             \\
            \phantom{\LARGE$\int$} Interaction function              & $\psi(x)$         & $\frac{1}{\sqrt{1 + x^2}}$                    \\
            \phantom{\LARGE$\int$} Interaction strength              & $K$               & 5                                             \\
            \phantom{\LARGE$\int$} Initial law                       & $\mfldis_0$       & $\mfldis^X_0 \otimes \mfldis^V_0$             \\
            \phantom{\LARGE$\int$} Initial law: $x$ marginal         & $\mfldis^X_0$     & $\mathcal U[-3, 3] \otimes \mathcal U[-3, 3]$ \\
            \phantom{\LARGE$\int$} Initial law: $v$ marginal         & $\mfldis^V_0$     & $\mathcal U[-1, 1] \otimes \mathcal U[-1, 1]$ \\
            \phantom{\LARGE$\int$} Time step                         & $\Delta t$        & 0.05                                          \\
            \phantom{\LARGE$\int$} Size for mean-field approximation & $J_{\infty}$      & 1000                                          \\
            \phantom{\LARGE$\int$} Number of independent simulations & $M$               & 400                                           \\
            \bottomrule
        \end{tabular}
        \caption{Parameters of the numerical experiment described in~\cref{sec:numerics}.}
        \label{table:training_params}
    \end{center}
\end{table}

The results of the numerical experiment are illustrated in~\cref{fig:evolution_errors} and~\cref{fig:final_error_X}.
\begin{itemize}
    \item
          \Cref{fig:evolution_errors} illustrates the evolution with time of $\expect \texttt{errX(t)}$ and $\expect \texttt{errV(t)}$.
          It appears clearly from the left panel of this figure that,
          for all the values of~$J$ considered,
          the quantity $\expect \texttt{errX(t)}$ tends to a constant value as time increases.
          The right panel shows that $\expect \texttt{errV(t)}$ tends to~0 for all the values of~$J$ considered,
          which is expected given that exponential flocking occurs for the parameters in~\cref{table:training_params}.
    \item
          \Cref{fig:final_error_X} illustrates the dependence on~$J$ of~$\expect \texttt{errX(T)}$,
          for a large fixed time~$T = 10$.
          As mentioned earlier, this expectation is estimated based on~$M$ independent realizations of the experiment.
          From these realizations, we also estimate the standard deviation of the Monte Carlo estimator for~$\expect \texttt{errX(T)}$,
          and include error bars in the figure corresponding to 2 standard deviations on each side of the expected value.
          The figure shows that~$\expect \texttt{errX(T)}$ is bounded from above independently of~$J$,
          and suggests that $\expect \texttt{errX(T)}$ tends to a constant for large~$J$,
          which is consistent with~\cref{theorem:uit_prop_of_chaos}.
\end{itemize}

\begin{figure}[ht]
    \centering
    \includegraphics[width=0.49\linewidth]{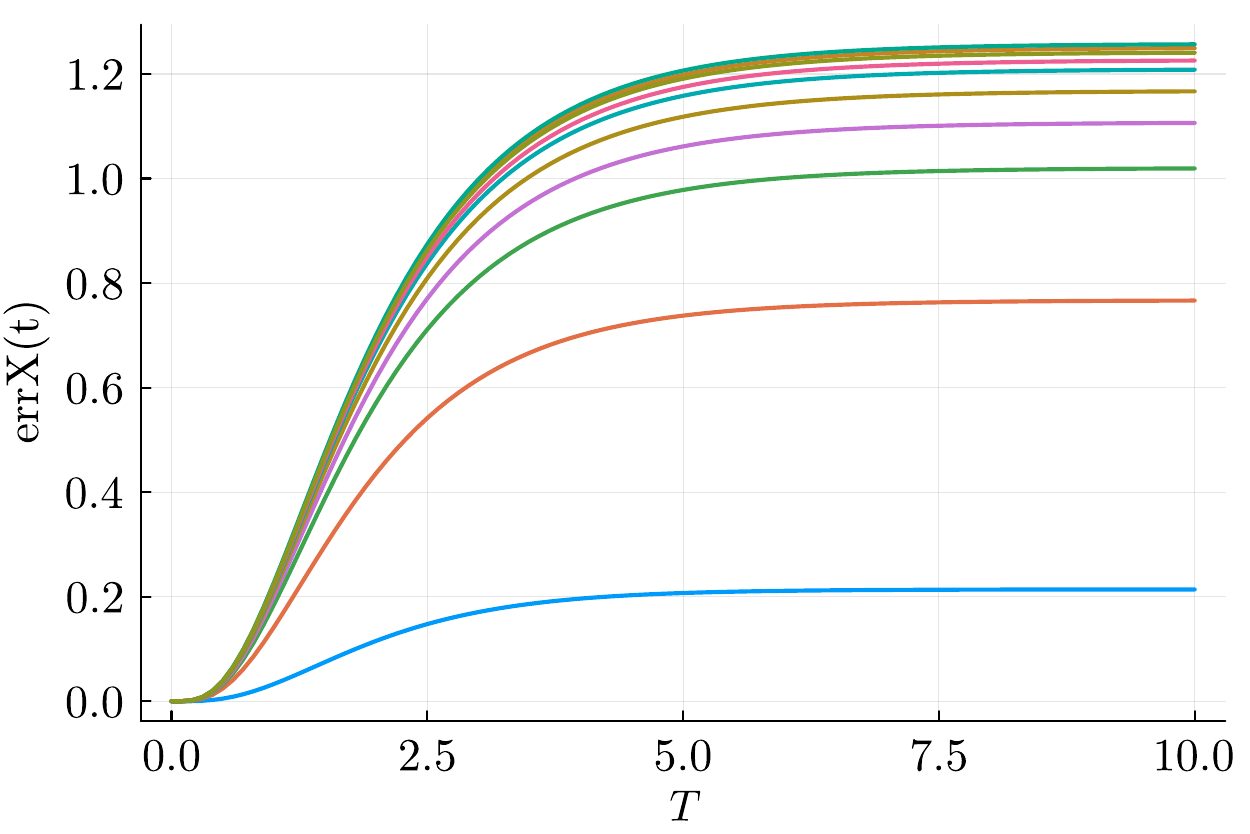}
    \includegraphics[width=0.49\linewidth]{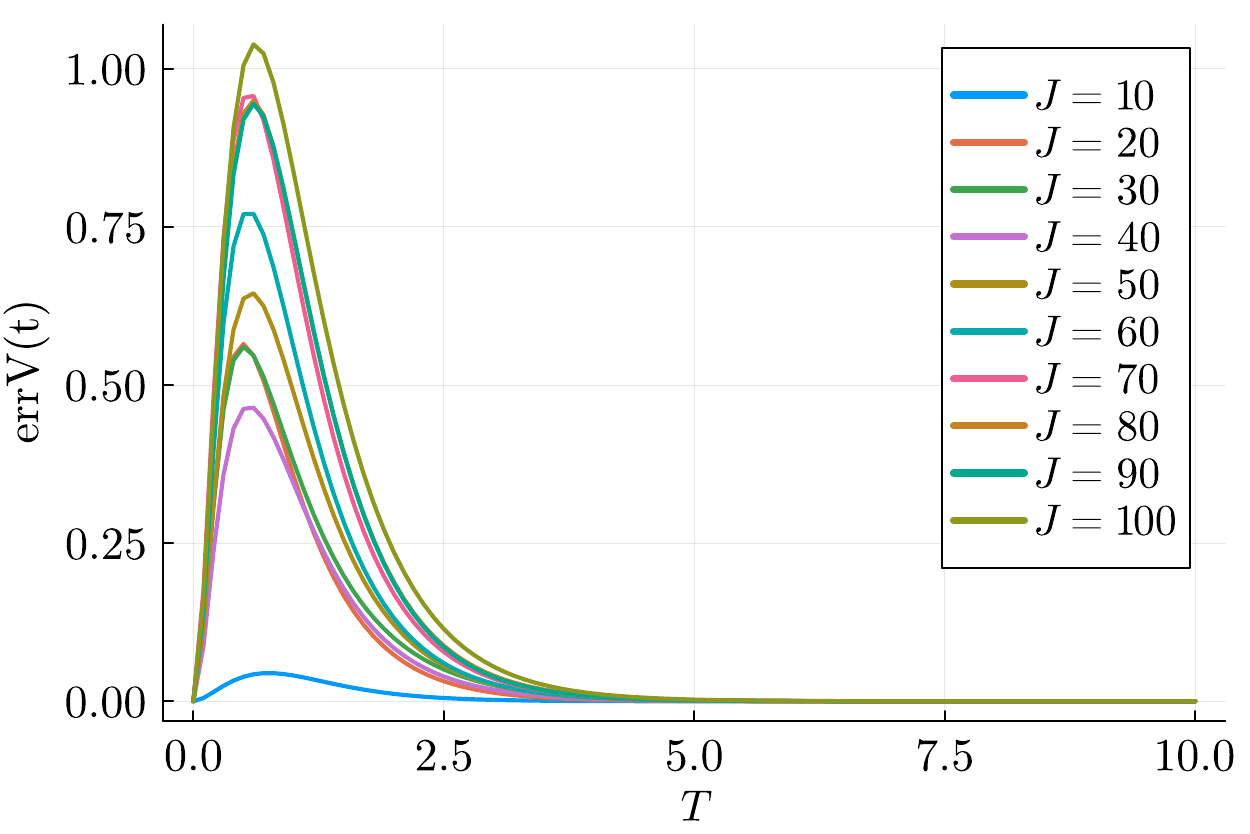}
    \caption{
        Evolution of~$\expect \texttt{errX(t)}$ and~$\expect \texttt{errV(t)}$,
        as defined in~\eqref{eq:errx} and \eqref{eq:errv},
        for various values of~$J$.
        We observe that $\expect \texttt{errX(t)}$ initially increases,
        but tends to a constant in the limit~$t \to \infty$ as flocking occurs.
    }
    \label{fig:evolution_errors}
\end{figure}

\begin{figure}[ht]
    \centering
    \includegraphics[width=0.49\linewidth]{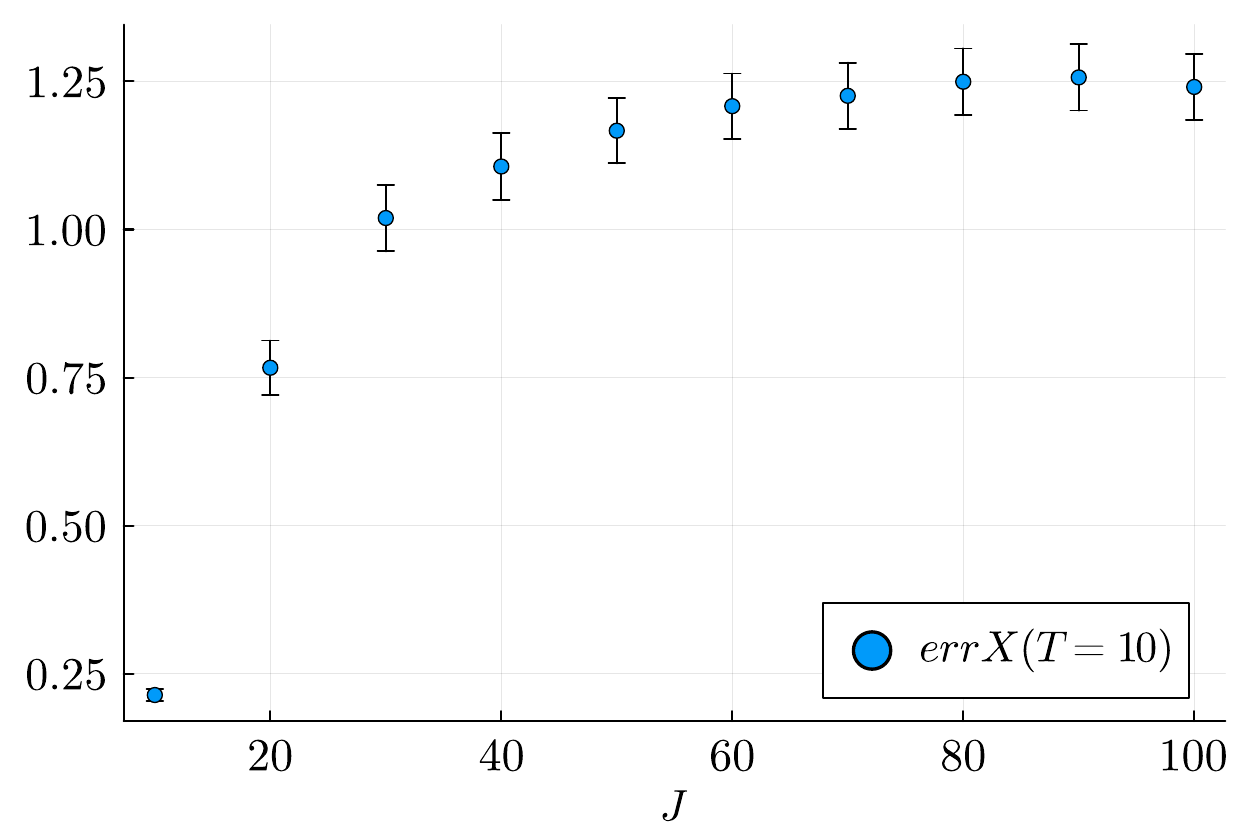}
    \caption{
        Value of~$\expect \texttt{errX(T)}$ for $T = 10$ and various values of the system size~$J$.
    }
    \label{fig:final_error_X}
\end{figure}

\paragraph{Acknowledgements}
The authors are grateful to Franca Hoffmann, Dohyeon Kim and Julien Reygner for useful discussions and suggestions. We also would like to thank the anonymous reviewers for making us aware of the references~\cite{ha2018firstorderreduction,HaKimPicklZhang2019probabilisticsingular,peszek2023heterogeneous} and for their helpful suggestions to improve the presentation of the paper.

UV is partially supported by the European Research Council (ERC) under the European Union's Horizon 2020 research and innovation programme (grant agreement No 810367),
and by the Agence Nationale de la Recherche under grants ANR-21-CE40-0006 (SINEQ) and ANR-23-CE40-0027 (IPSO).

\printbibliography
\end{document}